\documentclass[onefignum,onetabnum]{siamart220329}

\usepackage{lipsum}
\usepackage{amsfonts}
\usepackage{graphicx}
\usepackage{epstopdf}
\usepackage{algorithmic}
\usepackage{color,dsfont,bm} 
\usepackage[active]{srcltx}
\usepackage{cancel}
\usepackage{enumerate}
\usepackage{microtype,cite}

\ifpdf
  \DeclareGraphicsExtensions{.eps,.pdf,.png,.jpg}
\else
  \DeclareGraphicsExtensions{.eps}
\fi

\newsiamremark{remark}{Remark}
\newsiamremark{hypothesis}{Hypothesis}
\crefname{hypothesis}{Hypothesis}{Hypotheses}
\newsiamthm{claim}{Claim}

\newtheorem{example}{Example}[section]
\newtheorem{assumption}{Assumption}[section]

\headers{Convergence rate of a Dykstra-type algorithm}{XIAOZHOU Wang and TING KEI PONG}

\title{Convergence rate analysis of a Dykstra-type projection algorithm
\thanks{Submitted to the editors \today.
\funding{Ting Kei Pong is supported partly by the Hong Kong Research Grants Council PolyU153002/21p.}}}

\author{Xiaozhou Wang\thanks{The Hong Kong Polytechnic University, Hong Kong, People's Republic of China
  (\email{xzhou.wang@connect.polyu.hk}).}
\and   Ting Kei Pong\thanks{The Hong Kong Polytechnic University, Hong Kong, People's Republic of China
  (\email{tk.pong@polyu.edu.hk}).}
}

\usepackage{amsopn}

\def\R{{\rm I\!R}}

\def\Argmin{\mathop{\rm Arg\,min}}

\allowdisplaybreaks[4]
\numberwithin{equation}{section}

\definecolor{blue}{rgb}{0,0,0}

\ifpdf
\hypersetup{
}
\fi

\begin{document}

\maketitle

\begin{abstract}
Given closed convex sets $C_i$, $i=1,\ldots,\ell$, and some nonzero linear maps $A_i$, $i = 1,\ldots,\ell$,
of suitable dimensions, the multi-set split feasibility problem aims at finding a point in $\bigcap_{i=1}^\ell A_i^{-1}C_i$ based on computing projections onto $C_i$ and multiplications by $A_i$ and $A_i^T$.
In this paper, we consider the associated best approximation problem, i.e., the problem of computing projections onto $\bigcap_{i=1}^\ell A_i^{-1}C_i$;
we refer to this problem as the best approximation problem in multi-set split feasibility settings (BA-MSF).
We adapt the Dykstra's projection algorithm, which is classical for solving the BA-MSF in the special case when all $A_i = I$,
 to solve the general BA-MSF. Our Dykstra-type projection algorithm is derived by applying (proximal) coordinate gradient descent to the Lagrange dual problem,
 and it only requires computing projections onto $C_i$ and multiplications by $A_i$ and $A_i^T$ in each iteration. Under a standard relative interior condition and a genericity assumption on the point we need to project,
 we show that the dual objective satisfies the Kurdyka-{\L}ojasiewicz property with an explicitly computable exponent
 on a neighborhood of the (typically unbounded) dual solution set when each $C_i$ is $C^{1,\alpha}$-cone reducible for some $\alpha\in (0,1]$: this class of sets covers the class of $C^2$-cone reducible sets, which include all polyhedrons, second-order cone, and the cone of positive semidefinite matrices as special cases. Using this, explicit convergence rate (linear or sublinear) of the sequence generated by the Dykstra-type
 projection algorithm is derived. Concrete examples are constructed to illustrate the necessity of some of our assumptions.
\end{abstract}

\begin{keywords}
Dykstra's projection algorithm, Kurdyka-{\L}ojasiewicz property, linear convergence, $C^{1,\alpha}$-cone reducibility
\end{keywords}

\begin{MSCcodes}
90C25, 90C30, 90C46, 90C90
\end{MSCcodes}

\section{Introduction}

The multi-set split feasibility (MSF) problems, first introduced by Censor et al.  \cite{Censor-Elfving-Kopf-Bortfeld-2005}, are generalizations of the two-set split feasibility problems \cite{Censor-Elfving-1994} and convex feasibility problems  \cite{Bauschke-Borwein-1996}. These kinds of problems arise naturally in many contemporary application fields such as image reconstruction; see \cite{Censor-Elfving-Kopf-Bortfeld-2005} and references therein.
The MSF problem aims at finding a point in the intersection of the linear preimage of a collection of finitely many closed convex sets, i.e.,
\begin{align}\label{MSF}
\mbox{Find $x\in\R^n$ such that  $A_ix \in C_i$ for $i=1,\ldots, \ell$,}
\end{align}
where $C_i\subseteq \R^{m_i}$, $i = 1,\ldots,\ell$, are closed convex sets and $A_i\in \R^{m_i\times n}$ for each $i$; moreover, the projections onto $C_i$ are assumed to be easy to compute, while computing projections onto the sets $A_i^{-1}C_i$ can be difficult (see \cite{Censor-Elfving-Kopf-Bortfeld-2005}).
The assumptions concerning projections naturally call for iterative schemes that leverage projections onto $C_i$ for solving \eqref{MSF}.
One such scheme is the CQ-algorithm proposed in \cite{Byrne-2002} for solving the MSF problem with $\ell=2$, which was later generalized to solve general MSF problems in \cite{Censor-Elfving-Kopf-Bortfeld-2005}.
{\color{blue}Recent works on solution methods for MSF problems can be found in \cite{Brooke-Censor-Gibali-2020, Censor-Reem-Zaknoon-2022, Yang-Yang-Su-2014, Burke-Curtis-Wang-Wang-2015, Chambolle-Pock-2016}.} In this paper, we focus on a natural but relatively less studied variant of the MSF problem.
Specifically, given a $\bar v\in \R^n$, we consider the problem of finding the point in $\bigcap_{i=1}^\ell A_i^{-1}C_i$ that is closest to $\bar v$.
In other words, we consider the following best approximation problem in multi-set split feasibility settings (BA-MSF):
\begin{equation}\label{problem}
\min_{x\in \R^n}\, f(x):=\frac{1}{2}\|x-\bar{v}\|^2 \quad \text{s.t. $A_i x\in C_i,\, i=1,\ldots, \ell$},
\end{equation}
where  $\bar{v}\in \R^n$ is a given vector, each $A_i\in\R^{m_i\times n}\backslash\{0\}$ and each $C_i$ is a closed convex set; moreover, we assume that
\[
\bigcap_{i=1}^\ell A_i^{-1}C_i\neq \emptyset.
\]

When $A_i = I$ for all $i$, we refer to the corresponding problem~\eqref{problem} as the best approximation (BA) problem.
In this case, a classical solution method is the Dykstra's projection algorithm proposed in \cite{Boyle-Dykstra-1986,Han-1988}.
Each iteration of this algorithm only requires computing projections onto each $C_i$ instead of $\bigcap_{i=1}^\ell C_i$,
which can be advantageous because the  projection onto $\bigcap_{i=1}^\ell C_i$ can be more difficult to compute in general.
One remarkable feature of the Dykstra's projection algorithm is that the sequence generated will converge to the unique solution of the BA problem as long as $\bigcap_{i=1}^\ell C_i\neq \emptyset$;
see \cite{Boyle-Dykstra-1986}. This is in contrast to splitting methods such as Douglas-Rachford splitting, which typically requires additional assumptions to guarantee convergence; see \cite[Corollary~28.3]{Bauschke-Combettes-2017}.
The Dykstra's projection algorithm and its variants have
 been studied extensively in recent years concerning its convergence properties and relations to other algorithms;
see \cite{Bauschke-Borwein-1994, Bauschke-Burachik-Herman-Kaya-2020,  Bauschke-2000, Hundal-Deutsch-1997, Pang-2015, Pang-2019,Gaffke-Mathar-1989,Deutsch-Hundal-1994} and references therein. It is now known that the Dykstra's projection algorithm can be derived as a suitable application of the coordinate gradient descent method to the dual problem of \eqref{problem} with $A_i = I$ for all $i$; see \cite{Han-1988,Gaffke-Mathar-1989}.  Moreover, for BA problems with \emph{polyhedral} $C_i$, it is known that the sequence generated by the Dykstra's projection algorithm converges linearly \cite{Deutsch-Hundal-1994, Luo-Tseng-1993, Tseng-Yun-2009}.
On the other hand, convergence rates in the case of nonpolyhedral $C_i$ are not very well understood.

In this paper, we will adapt the classical Dykstra's projection algorithm to solve \eqref{problem} and analyze the convergence rate of the resulting algorithm. Following \cite{Gaffke-Mathar-1989, Han-1988}, we derive our Dykstra-type projection algorithm by applying (proximal) coordinate gradient descent to a suitable Lagrange dual problem of \eqref{problem}, and each iteration of our algorithm only requires projections onto $C_i$ and multiplications by $A_i$ and $A_i^T$.\footnote{Our algorithm reduces to the classical Dykstra's projection algorithm when $A_i = I$ for all $i$.} Then, by imposing the standard relative interior condition $\bigcap_{i=1}^\ell A_i^{-1}{\rm ri}\,C_i\neq \emptyset$ and a genericity assumption on the point $\bar v$, we show that the objective of the dual problem satisfies Kurdyka-{\L}ojasiewicz (KL) property with exponent $1/(\alpha+1)$ for some $\alpha \in (0,1]$ on a neighborhood of the (typically unbounded) dual solution set when each $C_i$ is \emph{$C^{1,\alpha}$-cone reducible}. Based on this, we establish the linear or sublinear convergence of the sequences generated by the Dykstra-type
 projection algorithm depending on the value of $\alpha\in (0,1]$. The key novelty of our work lies in \emph{both} the convergence rate analysis and the identification of the $C^{1,\alpha}$-cone reducibility condition:
\begin{itemize}
  \item First, our convergence rate results do not follow directly from standard KL-based convergence analysis frameworks such as those in \cite{Attouch-Bolte-2009,Attouch-Bolte-Redont-Soubeyran-2010,Attouch-Bolte-Svaiter-2013}. Indeed, the sequences generated by the Dykstra-type projection algorithm can be unbounded in general while the standard convergence rate analysis based on KL property typically requires the boundedness of the sequence generated. Our analysis is made possible thanks to the fact that, under the assumption $\bigcap_{i=1}^\ell A_i^{-1}{\rm ri}\,C_i\neq \emptyset$, the solution set of the dual problem can be written as the sum of a compact convex set and a {\em subspace}, a key fact which was also established in \cite{Auslender-Cominetti-Crouziex-1993} for BA problem, i.e., \eqref{problem} with each $A_i = I$.
  \item Second, we would like to point out that $C^{1,\alpha}$-cone reducible (with $\alpha\in (0,1]$) sets are prevalent in applications. Indeed, it covers the notion of ${C^2}$-cone reducible sets introduced in \cite{Shapiro-2003},\footnote{This notion was known as cone reducibility in \cite{Shapiro-2003}, and as $C^2$-cone reducibility in \cite[Definition~3.135]{Bonnans-Shapiro-2000}. Here, we adopt the latter terminology to highlight the differentiability property.} which contains sets such as polyhedrons, second-order cone and positive semidefinite cone.
     {\color{blue}In addition, as we will see later, the $p$-norm ball with $p\in (1,+\infty)$ can be shown to be $C^{1,\alpha}$-cone reducible with $\alpha = \min\{p-1,1\}$; see Example~\ref{c-1a-cone} below.
      Furthermore, when $p\in(1,2)$,  the $p$-norm ball is $C^{1,\alpha}$-cone reducible with $\alpha = p-1$  but not $C^2$-cone reducible (see Remark~\ref{remark-C1a}),
      showing that the class of $C^{1,\alpha}$-cone reducible sets is strictly larger than the class of $C^2$-cone reducible sets.}
      Our result is thus applicable to analyzing our Dykstra-type projection algorithm for a wide range of sets. Moreover, since our Dykstra-type projection algorithm reduces to the classical Dykstra's projection algorithm when $A_i = I$ for all $i$, we can deduce convergence rate results for the classical Dykstra's projection algorithm on a large class of \emph{nonpolyhedral} $C_i$, whose convergence rate was previously unknown.
\end{itemize}

The paper is organized as follows. Section~\ref{sec2} presents  some preliminary materials.
In section~\ref{sec3}, we adapt the classical Dykstra's projection algorithm to solve the BA-MSF problem \eqref{problem} and present some basic convergence properties.
Our key contributions are in sections~\ref{sec4} and \ref{sec5}, where we establish the KL property of the dual objective under suitable assumptions, and use that to study the convergence rate of our Dykstra-type projection algorithm. We also present concrete examples to show the indispensability of some of our assumptions in deriving the error bound and convergence rate results.

\section{Notation and preliminaries}\label{sec2}

In this paper, we use $\R^n$ to denote the $n$-dimensional  Euclidean space. For an $x\in \R^n$, we let $\|x\|_p=\sqrt[p]{|x_1|^p+\cdots+|x_n|^p}$ denote the $p$-norm, where $p\in [1,\infty)$; we also use $\|x\|$ to denote the 2-norm for notational simplicity. For $x$ and $y\in\R^n$, we use $\langle x,y \rangle$ to denote their inner product. For a matrix $A\in \R^{m\times n}$, we use $
\|A\|$ to denote its operator norm. We use $I$ to denote the identity matrix, whose dimension should be clear from the context. We also use $B(x,\eta)$  to  denote  a  closed  ball  centered  at $x$ with radius $\eta
\ge 0$,  i.e., $B(x,\eta) =\{ u:\| u - x\| \leq \eta\}$.

An extended-real-valued  function $h:\R^n\rightarrow (-\infty, \infty]$ is said to be proper if  ${\rm dom }\, h:=\{x\in\R^n : h(x)<\infty \}$ is nonempty.
A proper function is said to be closed if it is lower semicontinuous.
For  a proper  convex function $h$,  its subdifferential $\partial h(x)$ at an $x\in \R^n$  is defined as
\[
\partial h(x)=\{ \xi\in\R^n : h(y)\geq h(x)+\langle\xi, y-x\rangle\ \mbox{for all } y\in \R^n \}.
\]
The domain of $\partial h$ is defined as ${\rm dom}\,\partial h:=\{ x\in \R^n: \partial h(x)\neq \emptyset\}$, and the convex conjugate of $h$ is given by
\[
h^*(y)=\sup\{\langle x, y\rangle-h(x) : x\in \R^n  \}.
\]
For a proper closed convex function $h$, we have the following equivalences concerning $\partial h$ and $h^*$ (see, for example, \cite[Proposition 11.3]{Rockafellar-Wets-2009}):
\begin{align}\label{young}
h(x)+h^*(y)=\langle x, y\rangle \Longleftrightarrow y\in \partial h(x)\Longleftrightarrow x \in \partial h^*(y).
\end{align}
The proximal operator of a proper closed convex function $h$ at an $x\in \R^n$ is defined as
\[
{\rm prox}_{h}(x) := \Argmin_{u\in \R^n}\left\{\frac12\|u - x\|^2 + h(u)\right\};
\]
recall that the above set of minimizers is a singleton for every $x\in \R^n$, and ${\rm prox}_h:\R^n\to \R^n$ is nonexpansive. Finally, for a proper closed convex function $h$ with $\Argmin h \neq \emptyset$, the following inequality holds (see \cite[Proposition~10.59]{Rockafellar-Wets-2009}):
\begin{align}\label{VA}
h(x)-\inf h\leq {\rm dist}(0,\partial h(x)){\rm dist}(x, \Argmin h)\ \ \ \mbox{for all } x\in \R^n,
\end{align}
where ${\rm dist}(x,C):=\inf\{\|x-y\| : y\in C\}$ is the distance of an $x\in\R^n$ to a set $C$.

For a nonempty closed convex set $C\subseteq \R^n$, we use ${\rm int}\,C$ and ${\rm ri}\,C$ to represent the interior and relative interior  of $C$, respectively.
The indicator function and support function of such a $C$ are respectively defined as
\begin{align*}
\delta_C(x)=
\begin{cases}
  0 & \mbox{if $x\in C$,}  \\
  +\infty & \mbox{otherwise},
\end{cases}\quad\quad {\rm and}\quad\quad
 \sigma_C(x)=\sup\{\langle x, y\rangle : y\in C  \}.
\end{align*}
The normal cone of a nonempty closed convex set $C$ at $x\in C$ is defined as
\[
\mathcal{N}_C(x):= \partial \delta_C(x) = \{ v\in\R^n : \langle v, u-x\rangle \leq 0\ \mbox{for all } u\in C\},
\]
and the projection of  $x\in\R^n$ onto $C$ is denoted by ${\rm Proj}_C(x):= {\rm prox}_{\delta_C}(x)$.
For a closed convex cone $K\subseteq \R^n$, its polar is defined as $K^\circ=\{ u\in \R^n : \langle u, x\rangle \leq 0\ \mbox{for all } x\in K\}$.

Next, we recall the following definition of KL property, which is important for analyzing convergence rate of various first-order methods;  see, for example,
\cite{Attouch-Bolte-2009, Attouch-Bolte-Redont-Soubeyran-2010, Attouch-Bolte-Svaiter-2013, Bolte-Sabach-Teboulle-2014, Li-Pong-2018}.\!\!\!\!\!
\begin{definition}[KL property and exponent]
A proper closed convex function $h:\R^n\rightarrow (-\infty,\infty]$  is said to  satisfy the KL property at $\hat{x}\in{\rm dom}\,\partial h$ if there are  $c\in(0,\infty]$, a neighborhood
$U$ of $\hat{x}$ and a continuous concave function $\varphi:[0,c)\rightarrow [0,\infty)$ with $\varphi(0)=0$ such that
\begin{enumerate}[{\rm (i)}]
  \item $\varphi$ is continuously differentiable on $(0,c)$ and $\varphi'>0$ on $(0,c)$;
  \item $\varphi'(h(x)-h(\hat{x})){\rm dist}\,(0,\partial h(x))\geq 1$
whenever $x\in U \cap \{x: h(\hat{x})< h(x)< h(\hat{x})+c\}$.
\end{enumerate}
 If $h$ has KL property at $\hat{x}$ with the function $\varphi(v)=\alpha_0 v^{1-\theta}$ for some $\alpha_0 > 0$ and $\theta\in[0,1)$, then we say that $h$ satisfies the KL property at $\hat{x}$ with exponent $\theta$.

 A proper closed convex function $h$ satisfying the KL property with exponent $\theta\in [0,1)$ at every point in ${\rm dom}\,\partial h$ is called a KL function with
exponent $\theta$.
\end{definition}

KL property with exponent $\theta\in [0,1)$ is closely related to other notions of error bounds. For example, according to \cite[Theorem 5]{Bolte-Nguyen-Peypouquet-Suter-2017}, a proper closed convex function $h$ satisfies the KL property at an $\bar{x}\in \Argmin h$ with exponent $\theta\in[0,1)$ if and only if there exist positive constants ${c}_1$ and ${c}_2$ and a neighborhood  $U_{\bar{x}}$ of $\bar{x}$ such that
\[
c_1 {\rm dist}(x, \Argmin h)\le (h(x)-h(\bar{x}))^{1-\theta}\ \ \mbox{for all } x\in U_{\bar{x}} \cap \{x: h(\bar{x})< h(x)< h(\bar{x})+{c}_2\}.
\]

Before ending this section, we briefly review the (classical) Dykstra's projection algorithm proposed in \cite{Boyle-Dykstra-1986, Han-1988}.
The Dykstra's projection algorithm  was  developed to solve the following BA problem
\begin{equation}\label{problem-BA}
\min_{x\in \R^n}\frac{1}{2}\|x-\bar{v}\|^2 \quad \text{s.t. $x\in C_i,\, i=1,\ldots, \ell$},
\end{equation}
where each $C_i$ is a closed convex set and $\bigcap_{i=1}^\ell C_i\neq \emptyset$. Clearly, Problem~\eqref{problem-BA} is a special instance of \eqref{problem} with $A_i = I$ for all $i$. The Dykstra's projection algorithm is presented as Algorithm~\ref{alg-Dykstra-original} below. Note that it makes use of ${\rm Proj}_{C_i}$ in each iteration.
 \begin{algorithm}[h]
\caption{Dykstra's projection algorithm for \eqref{problem-BA}} \label{alg-Dykstra-original}
\begin{enumerate}[\textbf{Step} 1.]
  \item  Choose  $y^{0}_1=\cdots=y^0_\ell=0\in\R^n$ and $ x^0 = x^0_\ell=\bar{v}\in\R^n$.
  Set $t=0$.
  \item Set $x^{t+1}_0=x^t_\ell$. For $i=1,\ldots, \ell$, compute
\begin{equation*}\displayindent-15pt\displaywidth\textwidth
\begin{split}
& x^{t+1}_i= {\rm Proj}_{C_i}(y^t_i+x^{t+1}_{i-1}),\\
& y^{t+1}_i=y^t_i+ x^{t+1}_{i-1}-{\rm Proj}_{C_i}(y^t_i+x^{t+1}_{i-1}).
 \end{split}
\end{equation*}
  \item  Set ${x}^{t+1}=x^{t+1}_\ell$.  Update $t\leftarrow t+1$ and go to Step 2.
\end{enumerate}
\end{algorithm}

It was shown in \cite{Gaffke-Mathar-1989, Han-1988} that Algorithm~\ref{alg-Dykstra-original} is equivalent to a (proximal) coordinate gradient descent (CGD) method for solving the (negative of the) dual problem of \eqref{problem-BA}, which is given as
\begin{align*}
\min_{y_1,\ldots,y_\ell} {\frac{1}{2}\left\|\sum_{i=1}^\ell  y_i-\bar{v}\right\|^2-\frac{1}{2}\|\bar{v}\|^2}+\sum_{i=1}^\ell\sigma_{C_i}(y_i);
\end{align*}
Indeed, with $y^0_1=\cdots=y^0_\ell=0$, for $i=1,\ldots, \ell$, one can show that the $y^{t+1}_i$ in Algorithm \ref{alg-Dykstra-original} can be equivalently obtained  as
\begin{align*}
y^{t+1}_i := \Argmin_{y_i\in \R^{m_i}}\left\{ \left\langle \sum_{j=1}^{i-1}  y^{t+1}_j+\sum_{j=i}^{\ell}  y^{t}_j-\bar{v}, y_i-y^t_i\right\rangle+\frac{1}{2}\|y_i-y^t_i\|^2+\sigma_{C_i}(y_i)\right\}.
\end{align*}
The following theorem collects some known convergence results of Algorithm~\ref{alg-Dykstra-original}.
\begin{theorem}[Convergence properties of Algorithm \ref{alg-Dykstra-original}]\label{the-dykstra-ori}
Consider \eqref{problem-BA}.  Let $\{x^t\}$ be generated by Algorithm \ref{alg-Dykstra-original}. Then the following statements hold.
\begin{enumerate}[{\rm (i)}]
  \item It holds that $\lim_{t\rightarrow +\infty} x^t=x^*$, where $x^*$ is the unique solution of \eqref{problem-BA}.
  \item If each $C_i$ is polyhedral, then there exist $a_1>0$, $a_0\in (0,1)$ and a positive integer $\bar{t}$ such that
   \[
  \|x^t-x^*\|\leq a_1 a_0^t, \quad \forall\, t\geq \bar{t}.
  \]
\end{enumerate}
\end{theorem}
\begin{proof}
Item (i) was established in \cite{Boyle-Dykstra-1986}. Item (ii) can be deduced from \cite{Luo-Tseng-1993, Tseng-Yun-2009}; see also Theorem~\ref{the-rate-polyhedral} below.
\end{proof}

\section{A Dykstra-type projection algorithm for BA-MSF problems} \label{sec3}
In this section, we describe how  Algorithm \ref{alg-Dykstra-original} can be adapted to solve the BA-MSF problem \eqref{problem} and discuss some basic convergence properties of the resulting algorithm. We call the resulting algorithm a Dykstra-type projection algorithm.
From now on, for notational simplicity, we write
\begin{align}\label{A-D}
 {\bm A}\in\R^{m\times n} \mbox{ with } {\bm A}^T:=
 \begin{bmatrix}
   A_1^T& \cdots &A_\ell^T
 \end{bmatrix}  \quad \text{and}\quad D:=C_1\times\cdots\times C_\ell\subseteq\R^{m},
\end{align}
where $m=\sum_{i=1}^\ell m_i$; moreover, we denote the elements in $\R^m$ in boldface, i.e., ${\bm y} = (y_1,\ldots,y_\ell)$ with ${\bm y}\in \R^m$ and $y_i \in \R^{m_i}$ for $i = 1,\ldots,\ell$.

Following the development in \cite{Gaffke-Mathar-1989}, by considering a CGD method for solving the dual problem of \eqref{problem}, we present a Dykstra-type projection algorithm for solving the BA-MSF problem \eqref{problem}. To this end, {\color{blue}we first define $\psi:\R^{n+m}\rightarrow(-\infty,+\infty]$ as follows
\begin{align}\label{psi}
\psi (x,{\bm y}):=f(x)+\delta_D({\bm A}x+{\bm y}).
\end{align}
Then it holds that for every $\bar{x}\in \R^n$ and $\bar{{\bm y}}\in \R^m$
\begin{align}
\psi^*(\bar{x},\bar{{\bm y}})&=\sup_{x,{\bm y}}\{ \langle (\bar{x}, \bar{{\bm y}}), (x,{\bm y})\rangle-f(x)-\delta_D({\bm A}x+{\bm y}) \} \notag\\
&\!=\!\sup_{x,{\bm y}}\{ \langle -{\bm A}^T \bar{{\bm y}}\!+\! \bar{x}, x\rangle-f(x)+ \langle \bar{{\bm y}}, \mathbf{A}x\!+\!{\bm y} \rangle-\delta_D({\bm A}x\!+\!{\bm y}) \}\notag\\
&\!=\!\sup_{x}\{ \langle -{\bm A}^T \bar{{\bm y}}\!+\! \bar{x}, x\rangle-f(x)\}+ \sup_{{\bm z}}\{\langle \bar{{\bm y}}, {\bm z} \rangle-\delta_D({\bm z}) \}\notag\\
&\!=\! f^*(-{\bm A}^T \bar{{\bm y}}\!+\!\bar{x})+\sigma_D(\bar{{\bm y}}).   \label{psi-conjugate}
\end{align}
Therefore, in view of \cite[Theorem 31.2]{Rockafellar-1970} and \eqref{psi-conjugate}, the dual problem of \eqref{problem} is
 \begin{align}\label{problem-dual-psi}
  \max_{{\bm y}\in\R^m}   -\psi^*(0,{\bm y}).
  \end{align}
Since $f^*(x)=\frac{1}{2}\|x+\bar{v}\|^2-\frac{1}{2}\|\bar{v}\|^2$, \eqref{problem-dual-psi} can be further written as
  }
\begin{align}\label{problem-dual}
\max\limits_{{\bm y}\in\R^m} -{\frac{1}{2}\left\|\sum_{i=1}^\ell A_i^T y_i-\bar{v}\right\|^2+\frac{1}{2}\|\bar{v}\|^2}-\sum_{i=1}^\ell\sigma_{C_i}(y_i).
\end{align}
Note that using the definition of ${\bm A}$ and $D$ in \eqref{A-D}, we can rewrite the above dual problem equivalently as
\begin{align}\label{problem-duald}
d^*:= \min_{{\bm y}\in\R^m} d({\bm y}),\ \ \ {\rm where} \ d({\bm y}):=\underbrace{\frac{1}{2}\left\| {\bm A}^T {\bm y}-\bar{v}\right\|^2-\frac{1}{2}\|\bar{v}\|^2}_{g({\bm y})}+\sigma_{D}({\bm y}).
\end{align}
The following proposition states that the duality gap is always zero between \eqref{problem} and its dual problem \eqref{problem-dual}.
This justifies the rationale of solving \eqref{problem} by considering \eqref{problem-duald}.\!\!\!\!
\begin{proposition}\label{prop-primal-dual}
Consider \eqref{problem}. Let the function $d$ be given in \eqref{problem-duald}. Then it holds that ${\rm inf} \{f(x): x\in \bigcap_{i=1}^\ell A_i^{-1}C_i\}=-d^*< \infty$.
Moreover, for any ${\bm y}^*=(y_1^*,\ldots,y_\ell^*)\in \Argmin d$, it holds that
\[
x^*=\bar{v}-\sum_{i=1}^\ell A_i^Ty_i^*=\bar{v}-{\bm A}^T {\bm y}^*,
\]
where $x^*$ is the unique solution of \eqref{problem}.
\end{proposition}
\begin{proof}
{\color{blue}
For the $\psi$ defined in \eqref{psi}, notice that ${\bm y}\mapsto\psi(\cdot,{\bm y})$  is a proper closed convex bifunction.\footnote{{\color{blue}See \cite[Pages 291-292]{Rockafellar-1970} for the definition of a bifunction. The bifunction  ${\bm y}\mapsto\psi(\cdot,{\bm y})$ is said to be proper closed convex if its
graph function $(x,{\bm y})\mapsto \psi(x,{\bm y})$ is so; see \cite[Pages 293]{Rockafellar-1970}}} Moreover, the primal problem (i.e., $\min_{x\in \R^n}\psi(x,{\bm 0})$) has a unique solution thanks to $\bigcap_{i=1}^\ell A_i^{-1}C_i\neq \emptyset$. Then we can deduce from \cite[Theorem 30.4(i)]{Rockafellar-1970} that the duality gap is zero, i.e.,
\[
 \sup_{\bm y}-\psi^*(0,{\bm y})=\inf_x \psi(x,{\bm 0}) =\inf  \left\{ f(x): x\in \bigcap_{i=1}^\ell A_i^{-1}C_i\right\}=-d^*<\infty.
\]
 }

Next, if ${\bm y}^*=(y_1^*,\ldots,y_\ell^*)\in \Argmin d$, then we have from Theorem~10.1 and Exercise 8.8 of \cite{Rockafellar-Wets-2009} that
\[
{\bm 0} \in \partial d({\bm y}^*) = {\bm A}({\bm A}^T {\bm y}^* - \bar v) + \partial \sigma_D({\bm y}^*).
\]
The above display together with \cite[Example~11.4]{Rockafellar-Wets-2009} implies
\[
{\bm y^*}\in {\cal N}_D({\bm A}(\bar{v}-{\bm A}^T {\bm y}^*)).
\]
Write $\hat x:= \bar{v}-{\bm A}^T {\bm y}^*$ for notational simplicity. Then we have, in view of \cite[Proposition~10.5]{Rockafellar-Wets-2009}, that $y^*_i \in {\cal N}_{C_i}(A_i\hat x)$ for all $i$. Hence,
we have $
A_i^Ty_i^* \in A_i^T{\cal N}_{C_i}(A_i\hat x)\subseteq {\cal N}_{A_i^{-1}C_i}(\hat x)
$ for all $i$,
where the set inclusion follows directly from the definition of normal cone. Consequently, it holds that
\[
\bar v - \hat x = {\bm A}^T {\bm y}^* = \sum_{i=1}^\ell A_i^Ty_i^* \subseteq \sum_{i=1}^\ell {\cal N}_{A_i^{-1}C_i}(\hat x)\subseteq {\cal N}_{\bigcap_{i=1}^\ell A_i^{-1}C_i}(\hat x),
\]
where the last set inclusion follows directly from definition. The above display shows that $\hat x$ is the projection of $\bar v$ onto $\bigcap_{i=1}^\ell A_i^{-1}C_i$, and hence we have $\hat x = x^*$ as desired.
\end{proof}

Now, we derive the Dykstra-type projection algorithm by applying the (proximal) CGD method to solve \eqref{problem-duald}.
Set the starting point
\begin{align}\label{initial-point}
{\bm y}^0=(y_1^0,\ldots, y_\ell^0)=(0,\ldots,0).
\end{align}
For any $t\ge 0$ and any $i\in \{1,\ldots,\ell\}$, for notational simplicity, write
\begin{align}\label{y-tilde}
\tilde{{\bm y}}^{t+1}_{i-1}:=(y^{t+1}_1,\ldots,y^{t+1}_{i-1}, y^t_i,\ldots, y^t_\ell).
\end{align}
By applying the (proximal) CGD in a cyclic order to \eqref{problem-duald} with a proximal term induced by $\gamma_i I -  A_iA^T_i$, where ${\gamma}_i= \lambda_{{\rm max}}( A_i A^T_i )$, we obtain the following update formula for $y_i$, $i = 1,\ldots,\ell$:
 \begin{align}
 y^{t+1}_i :=& \Argmin_{y_i\in \R^{m_i}}\left\{ \langle \nabla_{y_i} g(\tilde{{\bm y}}^{t+1}_{i-1}), y_i-y^t_i\rangle+\frac{\gamma_i}{2}\|y_i-y^t_i\|^2+\sigma_{C_i}(y_i)\right\}. \label{coordinate}
 \end{align}
Following a similar  argument  in \cite[Pages 33-34]{Gaffke-Mathar-1989} and using \eqref{initial-point}-\eqref{coordinate},
one can derive the Dykstra-type projection algorithm for solving \eqref{problem} as we present in Algorithm \ref{alg-Dykstra}. In particular, one can show by induction that
\begin{equation}\label{xtdefinition}
x^t:= x^t_\ell=\bar{v}-{\bm A}^T{\bm y}^t\ \ \mbox{for all }t\ge 0.
\end{equation}
The detailed derivation is given in Appendix \ref{app-derivation}.

\begin{algorithm}[h]
\caption{Dykstra-type projection algorithm for  BA-MSF problem \eqref{problem}} \label{alg-Dykstra}
\begin{enumerate}[\textbf{Step} 1.]
  \item  Choose  $y^{0}_i=0\in\R^{m_i}$ for $i = 1,\ldots,\ell$, $x^0=\bar{v}\in\R^n$ and set ${\gamma}_i= \lambda_{{\rm max}}(A_i^T A_i)$ for $i=1,\ldots, \ell$.
  Set ${\bm y}^0=(y^0_1,\ldots,y^0_\ell)$,  $x^0_\ell=x^0$ and $t=0$.
  \item Set $x^{t+1}_0=x^t_\ell$. For $i=1,\ldots, \ell$, compute
\begin{equation}\displayindent-15pt\displaywidth\textwidth
\begin{split}
& x^{t+1}_i= (I-\gamma_i^{-1}A_i^T A_i )x^{t+1}_{i-1}+ \gamma_i^{-1}A_i^T {\rm Proj}_{C_i}(\gamma_i y^t_i+A_ix^{t+1}_{i-1}),\label{alg-1}  \\
& y^{t+1}_i=y^t_i+\gamma_i^{-1}A_i x^{t+1}_{i-1}-\gamma_i^{-1}{\rm Proj}_{C_i}(\gamma_i y^t_i+A_ix^{t+1}_{i-1}).
 \end{split}
  \end{equation}
  \item  Set ${x}^{t+1}\!\!=\!x^{t+1}_\ell\!$ and ${\bm y}^{t+1}\!\!=\!(y^{t+1}_1\!,\ldots,y^{t+1}_\ell)$.\!  Update $t\leftarrow t+1$ and go to Step 2.
\end{enumerate}
\end{algorithm}

One can see that  Algorithm~\ref{alg-Dykstra}  reduces to Algorithm \ref{alg-Dykstra-original} if each $A_i=I$.
As discussed above,  Algorithm~\ref{alg-Dykstra} is exactly a CGD method for solving \eqref{problem-duald}.
The following proposition collects some properties concerning the dual iterates $\{{\bm y}^t\}$ in Algorithm \ref{alg-Dykstra}, which are immediate consequences of known results on CGD; see, for example, \cite{Tseng-Yun-2009}.
We list them here together with their simple proofs for the convenience of our subsequent convergence rate analysis.
\begin{proposition}\label{prop-descent}
Consider \eqref{problem}. Let the function $d$ be  given in \eqref{problem-duald}   and $\{{\bm y}^t\}$ be the sequence generated by Algorithm \ref{alg-Dykstra}.
Then the following statements hold:
\begin{enumerate}[{\rm (i)}]
\item  For each $i = 1,\ldots,\ell$, it holds that
\[
\begin{aligned}
  d(\tilde{{\bm y}}^{t+1}_i)-d(\tilde{{\bm y}}^{t+1}_{i-1})
  &\le {\color{blue}\frac{1}{2}} \Delta_i^t\le -{\color{blue}\frac{\gamma_i}{2}}\|y^{t+1}_i - y^t_i\|^2,
\end{aligned}
\]
where $\Delta_i^t \!:=\! \langle  \nabla_{y_i} g(\tilde{{\bm y}}^{t+1}_{i-1}),y^{t+1}_i \!-\! y^t_i\rangle\!+\sigma_D(\tilde{{\bm y}}^{t+1}_i) \!- \sigma_D(\tilde{{\bm y}}^{t+1}_{i-1})$, {\color{blue} $\tilde{{\bm y}}^{t+1}_{i-1}$ is as in \eqref{y-tilde} for $i = 1,\ldots,\ell$  and  $\tilde{{\bm y}}^{t+1}_\ell:={{\bm y}}^{t+1}$.}\!
\item There exists $c > 0$ such that $d({\bm y}^{t+1})-d({\bm y}^t)\leq -c\|{\bm y}^{t+1}-{\bm y}^t\|^2$.
\item  $\lim_{t\to\infty}\|{\bm y}^{t+1}-{\bm y}^t\|=0$.
\item $\{{\bm y}^t\}$ is a stationary sequence, i.e., $\lim_{t\to\infty}{\rm dist}({\bm 0},\partial d({\bm y}^t))=0$.
\item Every accumulation point of $\{{\bm y}^t\}$ is a solution of  \eqref{problem-duald}.
\end{enumerate}
\end{proposition}
\begin{proof}
We first note from the definition of $\tilde{{\bm y}}^{t+1}_{i-1}$ in \eqref{y-tilde}, the definition of $y_i^{t+1}$ in \eqref{coordinate}, and the strong convexity of the objective in \eqref{coordinate} that for $i = 1,\ldots,\ell$
\begin{align}\label{eq:additional}
\!\!\langle  \nabla_{y_i} g(\tilde{{\bm y}}^{t+1}_{i-1}),y^{t+1}_i \!\!-\! y^t_i\rangle \!+ \!{\color{blue}\frac{\gamma_i}{2}}\|y^{t+1}_i\!\! -\! y^t_i\|^2\!\!+\!\sigma_D(\tilde{{\bm y}}^{t+1}_i)
\!\le\! \sigma_D(\tilde{{\bm y}}^{t+1}_{i-1}) \!-\! {\color{blue}\frac{\gamma_i}{2}}\|y^{t+1}_i\!\!\! -\! y^t_i\|^2\!\!.\!\!\!\!\!
\end{align}
In addition, if we let $z_i^t := \sum_{j=1}^{i-1} A_j^T y^{t+1}_j + \sum_{j=i+1}^{\ell} A_j^T y^t_j-\bar{v}$, then
\begin{align}
& d(\tilde{{\bm y}}^{t+1}_i) = g(\tilde{{\bm y}}^{t+1}_i) + \sigma_D(\tilde{{\bm y}}^{t+1}_i) = {\frac{1}{2}\left\|A_i^T y^{t+1}_i + z_i^t\right\|^2-\frac{1}{2}\|\bar{v}\|^2}+\sigma_D(\tilde{{\bm y}}^{t+1}_i)\notag\\
 &
 \overset{\rm (a)}= \!\frac{1}{2}\left\|A_i^T y^{t}_i + z_i^t\right\|^2\!-\!\frac{1}{2}\|\bar{v}\|^2\! +\!\langle  \nabla_{y_i} g(\tilde{{\bm y}}^{t+1}_{i-1}),y^{t+1}_i - y^t_i\rangle \!+\! \frac12\|A_i^T(y^{t+1}_i - y^t_i)\|^2\!+\!\sigma_D(\tilde{{\bm y}}^{t+1}_i)\notag\\
&   \le \frac{1}{2}\left\|A_i^T y^{t}_i + z_i^t\right\|^2\!-\!\frac{1}{2}\|\bar{v}\|^2\! + \!\langle  \nabla_{y_i} g(\tilde{{\bm y}}^{t+1}_{i-1}),y^{t+1}_i - y^t_i\rangle \!+\! \frac{\gamma_i}2\|y^{t+1}_i - y^t_i\|^2\!+\!\sigma_D(\tilde{{\bm y}}^{t+1}_i)\notag\\
&  \overset{\rm (b)}= d(\tilde{{\bm y}}^{t+1}_{i-1}) + \Delta_i^t + \frac{\gamma_i}2\|y^{t+1}_i - y^t_i\|^2{\color{blue}\overset{\rm (c)}\le} d(\tilde{{\bm y}}^{t+1}_{i-1}) - {\color{blue}\frac{\gamma_i}{2}}\|y^{t+1}_i - y^t_i\|^2,\notag
\end{align}
where (a) holds since $ A_i(z_i^t + A_i^Ty_i^t) = \nabla_{y_i}g(\tilde{{\bm y}}^{t+1}_{i-1})$, (b) follows from the definition of $\Delta_i^t$, and {\color{blue}(c)  follows from \eqref{eq:additional} and the definition of $\Delta^t_i$.}
This proves item (i).
Item (ii) is a direct consequence of (i), and item (iii) follows from item (ii) (since $\inf d>-\infty$ by Proposition~\ref{prop-primal-dual}).

For item (iv), from the first-order optimality condition of \eqref{coordinate}, we see that
\begin{align*}
0&\in \nabla_{y_i}g(\tilde{{\bm y}}^{t+1}_{i-1}) + \gamma_i (y^{t+1}_i-y^t_i)+\partial\sigma_{C_i}(y^{t+1}_i)\\
&= [\gamma_i (y^{t+1}_i-y^t_i)+(\nabla_{y_i}g(\tilde{{\bm y}}^{t+1}_{i-1})-\nabla_{y_i}g({{\bm y}}^{t+1}))]+\nabla_{y_i}g({{\bm y}}^{t+1})+\partial\sigma_{C_i}(y^{t+1}_i).
\end{align*}
Using the fact that ${\bm y}^{t+1}-{\bm y}^t\rightarrow 0$ from  item (iii), the {\em uniform} continuity of $\nabla g$ on $\R^m$ and the definitions  ${\bm y}^{t+1}=(y^{t+1}_1,\ldots, y^{t+1}_\ell)$ and $\tilde{{\bm y}}^{t+1}_{i-1}=(y^{t+1}_1,\ldots,y^{t+1}_{i-1}, y^t_i,\ldots, y^t_\ell)$ (see \eqref{y-tilde}), the above display implies that
$\{{\bm y}^t\}$ is a stationary sequence. Finally, item (v) follows immediately from item (iv) and the closedness of the subdifferential as a set-valued mapping \cite[Theorem 24.4]{Rockafellar-1970}.
\end{proof}

%

Proposition~\ref{prop-descent}(v) is useful only when accumulation points exist. In the case when each $A_i=I$ in \eqref{problem}, it was shown in \cite[Lemma 4.6]{Han-1988} that $d$ is level-bounded if
$\bigcap_{i=1}^\ell {\rm int}\,C_i\neq \emptyset$, which would further imply the boundedness of $\{{\bm y}^t\}$ and hence the existence of accumulation points.
However, in general, it can happen that $\bigcap_{i=1}^\ell {\rm int}\,C_i= \emptyset$ and the sequence $\{{\bm y}^t\}$ can be unbounded.
Fortunately, Proposition~\ref{prop-descent}(iv) states that $\{{\bm y}^t\}$ is \emph{always} a stationary sequence. Below, as in \cite{Auslender-Cominetti-Crouziex-1993}
which studied the case when $A_i = I$ for all $i$, we will show that every stationary sequence of the function $d$ in \eqref{problem-duald} is minimizing, under mild conditions on the set of intersection.
Recall from \cite[Definition~4.2.2]{Auslender-Alfred-Teboulle-2006} that proper closed convex functions with all stationary sequences being minimizing are said to be asymptotically well-behaved (AWB).

\begin{proposition}[AWB property of $d$]\label{prop-awb}
Consider \eqref{problem} and assume the condition $\bigcap_{i=1}^{\ell} A_i^{-1}{\rm ri}\,C_i \neq \emptyset$. Let $d$ be given in \eqref{problem-duald} and {\color{blue}$q({\bm y}):=\inf_x \psi(x,{\bm y})$ with $\psi$  defined in \eqref{psi}}. Then ${\bm 0}\in {\rm ri}\,{\rm dom}\,q$.
Moreover, it holds that $d({\bm y}) = q^*({\bm y})$ for all ${\bm y}$ and
\begin{align}\label{prop-awb-2}
\Argmin d = \partial q_{E}({\bm 0})+E^\bot \neq \emptyset,
\end{align}
where $q_{E}({\bm y}):=q({\rm Proj}_E({\bm y}))$ and
\begin{align}\label{prop-awb-3}
E:={\rm span}({\rm dom}\,q)={\rm span}(D- {\rm Range}({\bm A}))
\end{align}
with $D$ and ${\bm A}$ given in \eqref{A-D}.
Furthermore, every stationary sequence $\{{\bm z}^t\}$ of $d$ satisfies $d({\bm z}^t)\rightarrow d^*$ and
${\rm dist}({\bm z}^t,\Argmin d)\rightarrow 0$, where $d^*$ is given in \eqref{problem-duald}.
\end{proposition}
\begin{proof}
 From {\color{blue}the definition of $q$}, we have  $q({\bm y})<+\infty$ if and only if there exists $x\in \R^n$  such that $A_i x+y_i \in C_i$ for all $i$ (since ${\rm  dom}\, f=\R^n$), which implies that
${\rm dom}\, q=D- {\rm Range}({\bm A})$.
According to  \cite[Exercise 2.45]{Rockafellar-Wets-2009}, the above display implies that
\[
{\rm ri}\,{\rm dom}\, q={\rm ri}\, D-{\rm Range}({\bm A}).
\]
Thus, the assumption $\bigcap_{i=1}^{\ell} A_i^{-1}{\rm ri}\,C_i\neq \emptyset$ implies ${\bm 0}\in {\rm ri}\,{\rm dom}\, q$.
{\color{blue}
It follows from \eqref{psi-conjugate}, \eqref{problem-duald} and the definition of $q$ that} $d({\bm y})=\psi^*(0,{\bm y}) = q^*({\bm y})$ for all ${\bm y}$.
The desired conclusions now follow from \cite[Theorem~3.1]{Auslender-Cominetti-Crouziex-1993}.
\end{proof}


We can now state a global convergence result that is informative even when $\{{\bm y}^t\}$ is possibly unbounded. The proof follows directly from Propositions \ref{prop-descent}(iv) and \ref{prop-awb}.
\begin{theorem}[Global convergence without  boundedness condition]\label{the-convergence-awb}
Consider \eqref{problem} and assume that $\bigcap_{i=1}^{\ell} A_i^{-1}{\rm ri}\,C_i \neq \emptyset$. Let $\{{\bm y}^t\}$ be the sequence generated by  Algorithm \ref{alg-Dykstra}. Then it holds that
$d({\bm y}^t)\rightarrow d^*$ and ${\rm dist}({\bm y}^t, \Argmin d)\rightarrow 0$, where $d$ and $d^*$ are defined in \eqref{problem-duald}.
\end{theorem}

The above theorem states that the sequences $\{d({\bm y}^t)-d^*\}$ and $\{{\rm dist}({\bm y}^t, \Argmin d)\}$ converge to zero. In this paper, we are interested in their convergence rates. To this end, we recall that the KL property is widely used for analyzing the convergence rate of first-order methods, with the rate usually depending explicitly on the \emph{KL exponent} of a suitable potential function; see, for example, \cite{Attouch-Bolte-2009,Attouch-Bolte-Redont-Soubeyran-2010,Attouch-Bolte-Svaiter-2013,Li-Pong-2018,Yu-Li-Pong-2022}. In the next section, we identify a large class of sets $C_i$ such that the corresponding function $d$ can be shown to satisfy the KL property with an \emph{explicitly known exponent}. Moreover, we show that the KL exponent can be chosen to be \emph{uniform} over a neighborhood of the possibly {\em unbounded} $\Argmin d$. These results will be further developed in Section~\ref{sec5} to derive explicit convergence rates of $\{d({\bm y}^t)-d^*\}$ and $\{{\rm dist}({\bm y}^t, \Argmin d)\}$.

\section{KL property and $C^{1,\alpha}$-cone reducible sets}\label{sec4}

\subsection{$C^{1,\alpha}$-cone reducible sets}

In this subsection, we introduce a class of sets for our subsequent convergence rate analysis.
{\color{blue}Specifically, we define the following notion of $C^{1,\alpha}$-cone reducibility for a closed set ${\frak D}\subseteq \mathbb{X}$, which can be seen as a generalization of the notion of $C^{2}$-cone reducibility in \cite{Shapiro-2003};
here and throughout, we use the typefaces $\mathbb{X}$, $\mathbb{Y}$, $\mathbb{Z}$, etc., to denote finite-dimensional Hilbert spaces}
 with their associated inner products and norms denoted by $\langle\cdot,\cdot\rangle$ and $\|\cdot\|$, respectively, by an abuse of notation.
{\color{blue}We also let $\mathcal{T}^*$ denote the adjoint of a linear operator $\mathcal{T}: \mathbb{X}\rightarrow\mathbb{Z}$, defined via $\langle \mathcal{T}x, z\rangle=\langle x, \mathcal{T}^*z\rangle$ for all $x\in\mathbb{X}$ and $z\in\mathbb{Z}$}; and $\|{\cal T}\|$ is the operator norm of ${\cal T}$.

\begin{definition}[$C^{1,\alpha}$-cone reducible sets] \label{def-cone-reducible}
Let $\alpha\in (0,1]$.
A closed set $\mathfrak{D}\subseteq\mathbb{X}$ is said to be  $C^{1,\alpha}$-cone reducible at $\hat{x}\in\mathfrak{D}$ if there exist positive constant $\rho$,  a   closed convex pointed cone {\color{blue}$K$ in a   finite-dimensional Hilbert space $\mathbb{Y}$} and a mapping $\Xi:\mathbb{X}\rightarrow \mathbb{Y}$ that maps  $\hat{x}$ to $0$ and is continuously differentiable in $B(\hat{x},\rho)$  such that   $D\Xi(\hat{x})$ is surjective,\footnote{Recall that $D\Xi(x)$ is the derivative mapping of $\Xi$ at an $x\in\mathbb{X}$, which is the linear mapping from $\mathbb{X}$ to $\mathbb{Y}$ such that
$[D\Xi(x)](p)=\lim_{t \downarrow 0} \frac{\Xi(x+tp)-\Xi(x)}{t}$ for all $p\in\mathbb{X}$.}
the mapping $x\mapsto D\Xi(x)$ is $\alpha$-H\"{o}lder-continuous in $B(\hat{x},\rho)$ and
\[
\mathfrak{D}\cap B(\hat{x},\rho)=\{x: \Xi(x)\in K\}\cap B(\hat{x},\rho).
\]
A closed set $\mathfrak{D}$ is said to be  $C^{1,\alpha}$-cone reducible if it is $C^{1,\alpha}$-cone reducible   at every $x\in\mathfrak{D}$.
\end{definition}

Clearly, every $C^2$-cone reducible set (see \cite[Definition~3.1]{Shapiro-2003}) is $C^{1,\alpha}$-cone reducible for any $\alpha\in (0,1]$. According to \cite{Shapiro-2003}, all polyhedrons, the second-order cone and the cone of positive semidefinite matrices are $C^2$-cone reducible sets; thus, they are also $C^{1,\alpha}$-cone reducible for any $\alpha\in (0,1]$. The next lemma establishes a ``calmness-type" property for the normal cone mapping of a $C^{1, \alpha}$-cone reducible closed convex set. This can be seen as an extension of \cite[Theorem 4.4]{Yu-Li-Pong-2022}, which considered $C^{2}$-cone reducible sets. Our proof is also similar to that of \cite[Theorem 4.4]{Yu-Li-Pong-2022}.
\begin{lemma}\label{lemma-metric}
Let  $\alpha\in (0,1]$ and $\mathfrak{D}\subseteq\mathbb{X}$ be a  closed convex set that is   $C^{1,\alpha}$-cone reducible at $\hat{x}\in\mathfrak{D}$.
Let $\hat{y}\in\mathcal{N}_{\mathfrak{D}}(\hat{x})$.
Then there exist $\hat{\rho}>0$,  $\hat{\delta}>0$ and $\hat{\kappa}>0$ such that
\begin{align}\label{lemma-c-reducible}
\mathcal{N}_{\mathfrak{D}}(x)\cap B(\hat{y},\hat{\delta})\subseteq \mathcal{N}_{\mathfrak{D}}(\hat{x})+\hat{\kappa}\|x-\hat{x}\|^\alpha B(0,1) \ \ \mbox{for all }x\in B(\hat{x},\hat{\rho}).
\end{align}
\end{lemma}
\begin{proof}
Since $\mathfrak{D}$  is   $C^{1,\alpha}$-cone reducible at $\hat{x}$, there exist a positive constant ${\rho}$,  closed convex pointed cone $K\subseteq \mathbb{Y}$ and
 continuously differentiable mapping $\Xi:\mathbb{X}\rightarrow \mathbb{Y}$ satisfying the conditions in Definition \ref{def-cone-reducible} such that
\begin{align}\label{lemma-c-reducible-1-1}
 \mathfrak{D}\cap B(\hat{x},{\rho})=\{x: \Xi(x)\in K\}\cap B(\hat{x},{\rho}).
 \end{align}
 Since $D\Xi(\hat{x})$ is surjective, one can choose a positive  $\tilde{\rho}\in(0,\rho)$ such that  for any $x\in B(\hat{x},\tilde{\rho})$, the linear mapping $D\Xi({x})$ is  surjective (meaning that $D\Xi(x)D\Xi(x)^*$ is  invertible) and that
\begin{align}\label{lemma-c-reducible-1-2}
\sup_{x\in B(\hat x,\tilde \rho)}\|[D\Xi(x)D\Xi(x)^*]^{-1}\|=: \tau < \infty.
\end{align}

The  relation \eqref{lemma-c-reducible-1-1} means that
 \[
\delta_\mathfrak{D}(x)=\delta_K(\Xi(x))  \ \ \mbox{for all }x\in B(\hat{x},\tilde{\rho}).
 \]
 Using  the fact that $D\Xi({x})$ is surjective whenever $x\in B(\hat{x},\tilde{\rho})$, and invoking \cite[Exercise 10.7]{Rockafellar-Wets-2009}, we deduce from the above display that
 \begin{align}\label{lemma-c-reducible-1}
\mathcal{N}_\mathfrak{D}(x)=D\Xi({x})^*\mathcal{N}_K(\Xi(x))  \ \ \mbox{for all }x\in \mathfrak{D}\cap B(\hat{x},\tilde{\rho}).
 \end{align}
Fix any  $\tilde{\delta}>0$ and take
\begin{align}\label{lemma-c-reducible-1-3}
x\in \mathfrak{D}\cap B(\hat{x},\tilde{\rho}), \quad  y\in\mathcal{N}_\mathfrak{D}(x)\cap B(\hat{y},\tilde{\delta}).
\end{align}
 Relation \eqref{lemma-c-reducible-1} implies that there exists a $v_x\in\mathcal{N}_K(\Xi(x))$ so that
 \begin{align}\label{lemma-c-reducible-2}
 y=D\Xi({x})^*v_x.
 \end{align}
Recall that for any ${z}_0\in K$, it holds that
 \begin{align}\label{lemma-c-reducible-3}
 \mathcal{N}_K({z}_0)&\overset{\rm (a)}=\{ u: \langle u, z-{z}_0\rangle\leq 0, \forall\, z\in K\}\overset{\rm (b)}\subseteq\{ u: \langle u, z\rangle\leq 0, \forall\, z\in K\}\overset{\rm (c)}=K^\circ,
 \end{align}
 where (a) follows from the  definition of normal cone, and (b) and (c) hold because  $K$ is  a closed convex  cone.
In particular, equality holds throughout \eqref{lemma-c-reducible-3} if $z_0=0.$

Now, for the  $v_x$ in \eqref{lemma-c-reducible-2}, we have
 \begin{align}\label{lemma-c-reducible-4}
 \!D\Xi(\hat{x})^*v_x\!\overset{\rm (a)}\in\!  D\Xi(\hat{x})^*\mathcal{N}_K(\Xi(x)) \!\!\overset{\rm (b)}\subseteq\!\!  D\Xi(\hat{x})^*K^\circ \!\overset{\rm (c)}=\! D\Xi(\hat{x})^*\mathcal{N}_K(\Xi(\hat{x}))
 \!\overset{\rm (d)}=\! \mathcal{N}_\mathfrak{D}(\hat{x}),
 \end{align}
where (a) holds because $v_x\in\mathcal{N}_K(\Xi(x))$, (b) follows from $\Xi(x)\in K$ and \eqref{lemma-c-reducible-3},
(c) holds thanks to $\Xi(\hat{x})=0$, and (d) follows from \eqref{lemma-c-reducible-1}.

Next, for any $x$ and $y$ as in \eqref{lemma-c-reducible-1-3}, by   \eqref{lemma-c-reducible-2} and the surjectivity of  $D\Xi(x)$, we have
 \begin{align}\label{lemma-c-reducible-5}
v_x= [D\Xi(x)D\Xi(x)^*]^{-1}D\Xi(x)y=:H(x)D\Xi(x)y,
 \end{align}
where $H(x):=[D\Xi(x)D\Xi(x)^*]^{-1}$.

Let $L$ be the $\alpha$-H\"{o}lder continuity modulus of $D\Xi(\cdot)$ over the set  $x\in \mathfrak{D}\cap B(\hat{x},\tilde{\rho})$. Then, for any  $x$ and $y$ chosen as in \eqref{lemma-c-reducible-1-3},
we have
 \begin{align}
 \|v_x-v_{\hat{x}}\|&=\|[H(x)D\Xi(x)y-H(\hat{x})D\Xi(\hat{x})\hat{y}\|  \notag\\
 &\leq  \|H(x)D\Xi(x)y-H(\hat{x})D\Xi(\hat{x}){y}\|  + \|H(\hat{x})D\Xi(\hat{x})(y-\hat{y})\| \notag\\
 &\leq  \|H({x}) D\Xi(x)-H({x})D\Xi(\hat{x})+H({x})D\Xi(\hat{x})-H(\hat{x})D\Xi(\hat{x})\| \|{y}\|\notag \\
& \quad+ \|H(\hat{x})D\Xi(\hat{x})(y-\hat{y})\| \notag\\
& \leq (\|H({x})\| \|D\Xi(x)-D\Xi(\hat{x})\|+\|H({x})-H(\hat{x})\|\|D\Xi(\hat{x})\|) \|(y-\hat{y})+\hat{y}\|\notag\\
  &\quad+ \|H(\hat{x})D\Xi(\hat{x})\| \|y-\hat{y}\| \notag\\
 &\leq (\tau L \|x-\hat{x}\|^{\alpha}+2\tau \|D\Xi(\hat{x})\|)(\|\hat{y}\|+\|y-\hat{y}\|)+\tau\|D\Xi(\hat{x})\|\|y-\hat{y}\| \notag\\
 &\leq  (\tau L \tilde{\rho}^{\alpha}+2\tau \|D\Xi(\hat{x})\|)(\|\hat{y}\|+\tilde{\delta})+\tau\|D\Xi(\hat{x})\|\tilde{\delta}, \label{lemma-c-reducible-6}
 \end{align}
 where the first equality is due to \eqref{lemma-c-reducible-5}, the fourth inequality follows from \eqref{lemma-c-reducible-1-2} and the  $\alpha$-H\"{o}lder continuity of the mapping $D\Xi({\cdot})$,
  and the last inequality follows from \eqref{lemma-c-reducible-1-3}. Inequality \eqref{lemma-c-reducible-6} means that $v_x$ is bounded whenever $x$ and $y$ are chosen as in \eqref{lemma-c-reducible-1-3}, and we denote $\tilde{\kappa}:=\sup\|v_x\|<+\infty$.

 For any $x$ and $y$ chosen according to \eqref{lemma-c-reducible-1-3}, it holds that
 \[
 {\rm dist}(y,\mathcal{N}_{\mathfrak{D}}(\hat{x})) \leq \|y-D\Xi(\hat{x})^*v_x\|=\|D\Xi({x})^*v_x-D\Xi(\hat{x})^*v_x\|\leq L\tilde{\kappa}\|x-\hat{x}\|^\alpha,
 \]
 where the first inequality follows from \eqref{lemma-c-reducible-4}, the equality follows from \eqref{lemma-c-reducible-2}, and
  the last inequality follows from the $\alpha$-H\"{o}lder continuity of $D\Xi$ and the definition of $\tilde{\kappa}$. Letting $\hat{\rho}:=\tilde{\rho}, \hat{\delta}:=\tilde{\delta}$ and $\hat{\kappa}:=L\tilde{\kappa}$,
  the above display implies \eqref{lemma-c-reducible}.
\end{proof}

Equipped with Lemma~\ref{lemma-metric}, we are now ready to prove the main theorem of this subsection, which will be used for deriving KL exponent of the function $d$ given in \eqref{problem-duald} in the next subsection. It can be seen as an extension of \cite[Theorem 4.4]{Yu-Li-Pong-2022}, which studied $C^2$-cone reducible sets. Before stating the theorem, we first recall that a pair of closed convex sets $\{\mathfrak{D}_1, \mathfrak{D}_2\}$ is said to be boundedly linearly regular (see  \cite[Definition 5.6]{Bauschke-Borwein-1996}) at $\hat{x}\in\mathfrak{D}_1\cap \mathfrak{D}_2$ if for any  bounded neighborhood $\mathfrak{U}$ of $\hat{x}$,
there exists $c > 0$  such that
\[
{\rm dist}\,({x},\mathfrak{D}_1\cap \mathfrak{D}_2)\leq c({\rm dist}({x}, \mathfrak{D}_1)+{\rm dist}(x, \mathfrak{D}_2))\ \ \mbox{for all } x\in\mathfrak{U}.
\]
It is known that bounded linear regularity holds at any  $\hat{x}\in\mathfrak{D}_1\cap \mathfrak{D}_2$
when $\mathfrak{D}_1$ and $\mathfrak{D}_2$ are  polyhedral or when $\mathfrak{D}_1$ is polyhedral and $\mathfrak{D}_1 \cap {\rm ri}\,\mathfrak{D}_2\neq \emptyset$; see \cite[Corollary 3]{Bauschke-Borwein-Li-1999}.
\begin{theorem}\label{thm-kl-g}
Let $\mathfrak{D}\subseteq \mathbb{X}$ be a nonempty $C^{1,\alpha}$-cone reducible closed convex set with $\alpha\in(0,1]$, $\mathcal{T}: \mathbb{X}\rightarrow \mathbb{Y}$ be a linear mapping and   $l: \mathbb{Y}\rightarrow \R$ be strongly convex on any compact convex set with locally Lipschitz gradient, and $v\in\mathbb{X}$.
Consider
\[
h(x):=l(\mathcal{T}x)+\langle v,x\rangle+\sigma_\mathfrak{D}(x).
\]
Then, for any  $\bar{x}\in \Argmin h$, it holds that
\begin{align}\label{thm-kl-g-1}
\bar{x}\in\mathcal{N}_{\mathfrak{D}}(\bar{w})\ \text{with $\bar{w}:=-\mathcal{T}^*\nabla l(\mathcal{T}\bar{x})-v$}.
\end{align}
Moreover,   there exist   positive constants $\rho$,  $\delta$ and $\kappa$ such that
\begin{align}\label{thm-kl-g-2}
\mathcal{N}_{\mathfrak{D}}(w)\cap B(\bar{x},\delta)\subseteq \mathcal{N}_{\mathfrak{D}}(\bar{w})+\kappa\|w-\bar{w}\|^\alpha B(0,1)\ \ \mbox{for all }w\in B(\bar{w},\rho).
\end{align}
Furthermore, if  $\{\mathcal{T}^{-1}\{\mathcal{T}\bar{x}\}, \mathcal{N}_{\mathfrak{D}}(-\mathcal{T}^*\nabla l(\mathcal{T}\bar{x})-v)\}$ is boundedly linearly regular at $\bar x$, then $h$ satisfies the KL property at $\bar{x}$ with exponent $\frac{1}{\alpha+1}$.
\end{theorem}
\begin{proof}
The relation \eqref{thm-kl-g-1}  follows from the same argument as in \cite[Theorem 4.4]{Yu-Li-Pong-2022}. 
The relation  \eqref{thm-kl-g-2} follows from applying Lemma \ref{lemma-metric} to $(\bar{x},\bar{w})$ with $\bar{x}\in\mathcal{N}_{\mathfrak{D}}(\bar{w})$.

We now prove the alleged KL property at $\bar{x}$.
We first show that  \eqref{thm-kl-g-2} implies
\begin{align}\label{thm-kl-g-3}
{\rm dist}(x,(\partial\sigma_{\mathfrak{D}})^{-1}(\bar{w}))\leq \kappa {\rm dist}(\bar{w}, \partial \sigma_{\mathfrak{D}}(x)\cap B(\bar{w},\rho))^\alpha\ \ \mbox{for all }x\in B(\bar{x},\delta).
\end{align}
Our argument is similar to the proof of \cite[Theorem~3H.3]{Dontchev-Rockafellar-2009}.
Specifically, notice that the above display holds trivially if $\partial\sigma_{\mathfrak{D}}(x)\cap B(\bar{w},\rho)=\emptyset$.
Now, consider the case $\partial\sigma_{\mathfrak{D}}(x)\cap B(\bar{w},\rho)\neq \emptyset$.  Using $(\partial\sigma_{\mathfrak{D}})^{-1}=\mathcal{N}_{\mathfrak{D}}$ (see \cite[Example~11.4]{Rockafellar-Wets-2009}),
we see that
\[
w\in \partial \sigma_{\mathfrak{D}}(x)\cap B(\bar{w},\rho), x\in B(\bar{x},\delta)\Longleftrightarrow
w\in B(\bar{w},\rho), x\in \mathcal{N}_{\mathfrak{D}}(w)\cap B(\bar{x},\delta).
\]
Consequently, for any $x,w$ satisfying $w\in \partial \sigma_{\mathfrak{D}}(x)\cap B(\bar{w},\rho)$ and $x\in B(\bar{x},\delta)$, we have from the above display and \eqref{thm-kl-g-2} {\color{blue}that}
\[
{\rm dist}(x,(\partial\sigma_{\mathfrak{D}})^{-1}(\bar{w}))\leq \kappa\|w-\bar{w}\|^\alpha.
\]
Taking infimum with respect to $w\in \partial \sigma_{\mathfrak{D}}(x)\cap B(\bar{w},\rho)$, we obtain \eqref{thm-kl-g-3}.

In view of \cite[Exercise 3H.4]{Dontchev-Rockafellar-2009},
the inequality \eqref{thm-kl-g-3} implies that there exists $\delta'\in(0,\delta]$ such that
\begin{align}\label{thm-kl-g-4}
{\rm dist}(x,(\partial\sigma_{\mathfrak{D}})^{-1}(\bar{w}))\leq \kappa {\rm dist}(\bar{w}, \partial \sigma_{\mathfrak{D}}(x))^\alpha\ \ \mbox{for all }x\in B(\bar{x},\delta').
\end{align}

Now, define a proper closed convex function
\[
F(x):=\sigma_{\mathfrak{D}}(x)-\langle \bar{w}, x\rangle.
\]
Using \eqref{young}, the following equivalence holds
\begin{align}\label{thm-kl-g-5}
z\in (\partial\sigma_{\mathfrak{D}})^{-1}(\bar{w})\Leftrightarrow\bar{w}\in\partial\sigma_{\mathfrak{D}}(z)\Leftrightarrow0\in \partial F(z)\Leftrightarrow z\in (\partial F)^{-1}(0).
\end{align}
Using \eqref{thm-kl-g-5}, we can rewrite \eqref{thm-kl-g-4}  as
\begin{align}\label{thm-kl-g-6}
{\rm dist}(x,(\partial F)^{-1}(0))\leq \kappa {\rm dist}(0, \partial F(x))^\alpha\ \ \mbox{for all }x\in B(\bar{x},\delta').
\end{align}
Then, we have for any $x\in B(\bar{x},\delta')$ that
\begin{align*}
F(x)-F(\bar{x})
\leq {\rm dist}(0, \partial F(x)) {\rm dist}(x,(\partial F)^{-1}(0))\leq \kappa{\rm dist}(0, \partial F(x))^{1+\alpha},
\end{align*}
where the first inequality follows from \eqref{VA},\footnote{Note that $(\partial F)^{-1}(0)\neq \emptyset$ since it contains $\bar x$ in view of \eqref{thm-kl-g-1}, \eqref{thm-kl-g-5} and the fact that $(\partial\sigma_{\mathfrak{D}})^{-1}=\mathcal{N}_{\mathfrak{D}}$ (see \cite[Example~11.4]{Rockafellar-Wets-2009}).} the second inequality  is due to \eqref{thm-kl-g-6}.
The above display implies  that $F$ satisfies the KL property at $\bar{x}$ with exponent $\frac{1}{1+\alpha}$, which by \cite[Theorem 5]{Bolte-Nguyen-Peypouquet-Suter-2017} also implies the existence of $\bar\kappa>0$ such that
\begin{align*}
{\rm dist}(x, (\partial F)^{-1}(0))\leq \bar{\kappa}(F(x)-F(\bar{x}))^{1-\frac{1}{1+\alpha}}\ \ \mbox{for all }x\in B(\bar{x},\delta').
\end{align*}
Using the definition of $F$ and \eqref{thm-kl-g-5}, the above display can be equivalently written as
\begin{align}\label{thm-kl-g-7}
{\rm dist}(x, \mathcal{N}_{\mathfrak{D}}(\bar{w}))&={\rm dist}(x, (\partial\sigma_{\frak D})^{-1}(\bar{w})) = {\rm dist}(x, (\partial F)^{-1}(0))\notag\\
&\leq \bar{\kappa}(\sigma_{\mathfrak{D}}(x)-\sigma_{\mathfrak{D}}(\bar{x})-\langle\bar{w}, x-\bar{x}\rangle)^{1-\frac{1}{1+\alpha}}\ \ \mbox{for all }x\in B(\bar{x},\delta'),
\end{align}
where the first equality follows from \cite[Example~11.4]{Rockafellar-Wets-2009}.
Next, notice that we have $\mathcal{T}x = \mathcal{T}\bar x$ for all $x\in \Argmin h$ thanks to the strict convexity of $l$,\footnote{{\color{blue}Indeed, suppose that $\mathcal{T} x \neq \mathcal{T} \bar{x}$ for some $x \in{\rm Argmin}\, h$. Then we have the following  contradiction:
\begin{align*}
\inf h = (h(x)+h(\bar{x}))/2> l((\mathcal{T}x+\mathcal{T}\bar{x})/2)+\langle v, (x+\bar{x})/2\rangle+\sigma_{\mathfrak{D}}((x+\bar{x})/2) \geq \inf h,
\end{align*}
where the first inequality follows from the strict convexity of $l$, convexity of $\langle v,\cdot\rangle+\sigma_{\mathfrak{D}}(\cdot)$ and  the assumption that $\mathcal{T} x \neq \mathcal{T} \bar{x}$, and the second inequality follows from
the fact that $(x+\bar{x})/2 \in {\rm dom}\, h$ as $h$ is convex.}
} and moreover
\begin{align*}
\Argmin h&=\{x: 0\in \partial h(x)\}=\{x:  \mathcal{T}x=\mathcal{T}\bar{x}, x\in \mathcal{N}_{\mathfrak{D}}(-\mathcal{T}^*\nabla l(\mathcal{T}\bar{x})-v)\}\\
&=\{x:  \mathcal{T}x=\mathcal{T}\bar{x}, x\in \mathcal{N}_{\mathfrak{D}}(\bar{w})\}.
\end{align*}
Also, for any bounded convex neighborhood $U\subseteq  B(\bar{x},\delta')$ of $\bar{x}$, we have for any $x\in U$ that
\begin{align}
\|\mathcal{T}x-\mathcal{T}\bar{x}\|^{\frac{1+\alpha}{\alpha}}
&=\|\mathcal{T}x-\mathcal{T}\bar{x}\|^{\frac{1}{\alpha}-1}\|\mathcal{T}x-\mathcal{T}\bar{x}\|^{2}\notag \leq
\bigg(\sup_{u\in U}\|\mathcal{T}u-\mathcal{T}\bar{x}\|^{\frac{1}{\alpha}-1}\bigg)\|\mathcal{T}x-\mathcal{T}\bar{x}\|^{2} \notag\\
&\leq \bar{M}(l(\mathcal{T}x)-l(\mathcal{T}\bar{x})-\langle \mathcal{T}^*\nabla l(\mathcal{T}\bar{x}),x-\bar{x}\rangle),   \label{thm-kl-g-8}
\end{align}
where the last inequality holds for some $\bar M > 0$ thanks to the strongly convexity of $l$ on $U$ and the fact that $\sup_{u\in U}\|\mathcal{T}u-\mathcal{T}\bar{x}\|^{\frac{1}{\alpha}-1}<\infty$ (thanks to $\alpha\in (0,1]$).
Now, using the bounded linear regularity condition and the Hoffman's error bound, there exist $\beta_1> 0$ and $\beta_2>0$ such that whenever $x\in U$, we have
\begin{align}
&{\rm dist}(x,\mathcal{T}^{-1}\{\mathcal{T}\bar{x}\}\cap  \mathcal{N}_{\mathfrak{D}}(\bar{w}) )\leq \beta_1[{\rm dist}(x,\mathcal{T}^{-1}\{\mathcal{T}\bar{x}\})+ {\rm dist}(x,\mathcal{N}_{\mathfrak{D}}(\bar{w}))], \label{thm-kl-g-9}\\
&{\rm dist}(x,\mathcal{T}^{-1}\{\mathcal{T}\bar{x}\})\leq \beta_2\|\mathcal{T}x-\mathcal{T}\bar{x}\|. \label{thm-kl-g-10}
\end{align}
Then, it follows that for any $x\in U$
\begin{align}
{\rm dist}(x,\Argmin h)&={\rm dist}(x,\mathcal{T}^{-1}\{\mathcal{T}\bar{x}\}\cap  \mathcal{N}_{\mathfrak{D}}(\bar{w}))\!\leq\! \beta_1[\beta_2\|\mathcal{T}x-\mathcal{T}\bar{x}\|+ {\rm dist}(x,\mathcal{N}_{\mathfrak{D}}(\bar{w}))] \notag\\
&\leq \beta_1[\beta_2  \bar{M}^{\frac{\alpha}{\alpha+1}}(l(\mathcal{T}x)-l(\mathcal{T}\bar{x})-\langle \mathcal{T}^*\nabla l(\mathcal{T}\bar{x}),x-\bar{x}\rangle)^{1-\frac{1}{\alpha+1}}\notag\\
  &\quad+ \bar{\kappa}(\sigma_{\mathfrak{D}}(x)-\sigma_{\mathfrak{D}}(\bar{x})-\langle\bar{w}, x-\bar{x}\rangle)^{1-\frac{1}{\alpha+1}}],
   \label{thm-kl-g-11}
\end{align}
where the first inequality follows from \eqref{thm-kl-g-9}-\eqref{thm-kl-g-10}, and the last inequality follows from \eqref{thm-kl-g-7} and \eqref{thm-kl-g-8}.
Finally, notice that
\begin{equation}\label{thm-kl-g-12}
h(x)-h(\bar{x})=l(\mathcal{T}x)-l(\mathcal{T}\bar{x})-\langle \mathcal{T}^*\nabla l(\mathcal{T}\bar{x}),x-\bar{x}\rangle+\sigma_{\mathfrak{D}}(x)-\sigma_{\mathfrak{D}}(\bar{x})-\langle \bar{w},x-\bar{x}\rangle.
\end{equation}
Combining \eqref{thm-kl-g-11}-\eqref{thm-kl-g-12} and the inequality $a^p+b^p\leq 2^{1-p}(a+b)^p$ for any $a,b\geq 0$ and $p\in(0,1]$, we deduce that there exists positive constant $c$ such that
\[
{\rm dist}(x,\Argmin h)\leq  c(h(x)-h(\bar{x}))^{{1-\frac{1}{\alpha+1}}}\ \ \mbox{for all }x\in U.
\]
The desired result now follows upon invoking \cite[Theorem 5]{Bolte-Nguyen-Peypouquet-Suter-2017}.
\end{proof}

Verifying $C^{1,\alpha}$-cone reducibility directly from the definition can be nontrivial. 
The following proposition is handy for checking whether a set is $C^{1,\alpha}$-cone reducible. It is an analogue of \cite[Proposition 3.2]{Shapiro-2003}, which studied $C^2$-cone reducible sets.
\begin{proposition}\label{prop:C11}
Let $V:=\{ x : G(x)\in K\}$, where $G:\mathbb{X}\rightarrow \mathbb{Y}$ is continuously differentiable with locally $\alpha$-H\"older-continuous derivative for some $\alpha\in (0,1]$, and $K\subseteq \mathbb{Y}$ be a closed convex set.
If $K$ is $C^{1,\alpha}$-cone reducible at $y_0=G(x_0)$ and
\[
D G(x_0)\mathbb{X}+{\rm lin}\, T_K(y_0)=\mathbb{Y},
\]
where $T_K(y_0)$ is the tangent cone of $K$ at $y_0$ and ${\rm lin}\, T_K(y_0):=T_K(y_0)\cap -T_K(y_0)$. Then $V$ is $C^{1,\alpha}$-cone reducible at $x_0$.
\end{proposition}
\begin{proof}
The proof follows the same argument as in \cite[Proposition 3.2]{Shapiro-2003}.
\end{proof}

When $\mathbb{X}=\R^n$, $\mathbb{Y}=\R^m$, $K=-\R^m_+$ and $G(x)=(g_1(x),\ldots, g_m(x))$ with each $g_i$ being continuously differentiable with locally $\alpha$-H\"older-continuous derivative, since $K=-\R^m_+$ is $C^2$-cone reducible \cite{Shapiro-2003}, we deduce from Proposition~\ref{prop:C11} that the set
\[
V:=\{ x\in \R^n : g_i(x)\leq 0, i=1,\ldots,m \}
\]
is $C^{1,\alpha}$-cone reducible at $x_0$ if  $\{\nabla g_i(x_0): i\in {\cal I}\}$ is linearly independent, where ${\cal I}:=\{ i : g_i(x_0)=0\}$.
As further concrete examples utilizing Proposition~\ref{prop:C11}, we show below that the $p$-cone and the $p$-norm ball with $p\in (1,\infty)$ are $C^{1,\alpha}$-cone reducible for some $\alpha\in (0,1]$. Recall that the $p$-cone $\mathcal{K}_p^{n+1}\subseteq \R^{n+1}$ is defined as
\[
\mathcal{K}_p^{n+1}:=\{(x,r)\in\R^{n}\times \R: r\geq \|x\|_p \}.
\]

\begin{example}[$\mathcal{K}_p^{n+1}$, $p \in (1,\infty)$, is $C^{1,\alpha}$-cone reducible]\label{c-1a-cone}
Notice that ${\cal K}_p^{n+1} = \{(x,r)\in\R^{n}\times \R: G(x,r)\le 0\}$ with $G(x,r):=\|x\|_p-r$.
As pointed out in \cite{Shapiro-2003}, any closed convex cone is $C^2$-cone reducible at the origin and its relative interior, and the corresponding $\Xi$ can be chosen to be linear.
So, we only need to show that $\mathcal{K}_p^{n+1}$ is $C^{1,\alpha}$-cone reducible  at every nonzero boundary point.
Note that when $x\neq 0$, we have at any $(x,r)$ satisfying $r\ge \|x\|_p$ that
\[
\nabla G(x,r)=
\|x\|_p^{1-p}\cdot
\begin{bmatrix}
  {\rm sgn}(x_1)|x_1|^{p-1} & \cdots & {\rm sgn}(x_n)|x_n|^{p-1} &  -\|x\|_p^{p-1}
\end{bmatrix}^T.
\]
Thus, for any $(x,r)\neq 0$ on the boundary of $\mathcal{K}_p^{n+1}$, one has $r = \|x\|_p$ and $\nabla G(x,\|x\|_p)\neq 0$; hence, $\{\nabla G(x,\|x\|_p)\}$ is linearly independent.

From the discussion preceding this example, it now remains to show that $\nabla G$ is locally H\"{o}lder continuous at any nonzero boundary point of $\mathcal{K}_p^{n+1}$. Using the display above, it is routine to check that $\nabla G$ is locally $(p-1)$-H\"{o}lder continuous when $p \in (1,2)$ and is locally Lipschitz continuous (and hence $1$-H\"{o}lder continuous) when $p \in [2,\infty)$.
Therefore, $\mathcal{K}_p^{n+1}$ is $C^{1,\alpha}$-cone reducible for any $p\in(1,+\infty)$ with $\alpha=\min\{1, p-1\}$.
\end{example}
\begin{example}[The $p$-norm ball, $p \in (1,\infty)$, is $C^{1,\alpha}$-cone reducible]\label{c-1a-ball}
Similar to the discussion in Example \ref{c-1a-cone}, one can show that the $p$-norm ball  $\{ x \in \R^n : \|x-x_0\|_p \leq \beta\}$ (center at $x_0$ with radius $\beta>0$) is  $C^{1,\alpha}$-cone reducible with  $\alpha=\min\{1, p-1\}$.
\end{example}

\subsection{KL properties of the function $d$ in \eqref{problem-duald}}
In this subsection, we show that the function $d$ in \eqref{problem-duald} satisfies the KL property with an explicit exponent under suitable assumptions on \eqref{problem}.
Specifically, we consider the following assumption.
\begin{assumption}\label{assu-kl}
Consider \eqref{problem}. Suppose the following conditions hold.
\begin{enumerate}[{\rm (i)}]
\item Each $C_i$ is a $C^{1,\alpha}$-cone reducible closed convex set with $\alpha \in(0,1]$;\footnote{Using this, one can verify directly from definition that $D$ is $C^{1,\alpha}$-cone reducible for the same $\alpha$.}
\item $\bigcap_{i=1}^{\ell} A_i^{-1}{\rm ri}\,C_i \neq \emptyset$;
\item $0\in x^*-\bar{v}+ {\rm ri}\, \partial (\sum_{i=1}^\ell \delta_{A_i^{-1}C_i })(x^*)$, where $x^*$ is the unique solution of \eqref{problem}.
\end{enumerate}
\end{assumption}

The following proposition gives the KL exponent of  the function $d$ in \eqref{problem-duald}.
\begin{proposition}[KL exponent of $d$]\label{prop-kl}
Consider  \eqref{problem}.  
Suppose that Assumption~\ref{assu-kl} holds.
Then the function $d$ in \eqref{problem-duald} is a  KL function with exponent $\frac{1}{\alpha+1}$.
\end{proposition}
\begin{proof}
Fix any ${\bm y}^*=(y_1^*,\ldots, y_\ell^*)\in \Argmin d$.
Then Proposition~\ref{prop-primal-dual} implies that
\begin{align}\label{prop-kl-1}
x^*=\bar{v}-{\bm A}^T{\bm y}^*,
\end{align}
where $x^*$ is the unique solution of \eqref{problem}.

Recall from \eqref{problem} that $f(x)=\frac{1}{2}\|x-\bar{v}\|^2$.
We first  show that
\begin{align}\label{prop-kl-2}
\{({\bm A^T})^{-1}\{{\bm A}^T{\bm y}^*\},\mathcal{N}_D(-{\bm A}\nabla f({\bm A}^T{\bm y}^*))\}
\end{align}
is boundedly linearly regular at ${\bm y}^*$.
To this end, we first observe that
\begin{align}\label{prop-kl-3}
({\bm A^T})^{-1}\{{\bm A}^T{\bm y}^*\}=(y_1^*,\ldots,y_\ell^*)+{\rm ker}\,{\bm A}^T.
\end{align}
Moreover, using \eqref{prop-kl-1} and noting $\nabla f({\bm A}^T{\bm y}^*) = {\bm A}^T{\bm y}^* - \bar v$, we see that
\begin{equation}
\mathcal{N}_D(-{\bm A}\nabla f({\bm A}^T{\bm y}^*))=\mathcal{N}_D({\bm A}x^*)
=\mathcal{N}_{C_1}(A_1x^*) \times \cdots\times \mathcal{N}_{C_\ell}(A_\ell x^*). \label{prop-kl-4}
\end{equation}
Furthermore, we deduce using Assumption~\ref{assu-kl}(iii) that
\begin{align*}
\textstyle
\bar v - x^*\!\in\!  {\rm ri}\, \partial\left(\sum_{i=1}^\ell\delta_{A_i^{-1}C_i}\right)(x^*)\!=\! {\rm ri}\, \left(\sum_{i=1}^\ell A_i^T \mathcal{N}_{C_i}(A_i x^*)\right)\!=\!\sum_{i=1}^\ell A_i^T {\rm ri}\, \mathcal{N}_{C_i}(A_i x^*),
\end{align*}
where the first equality follows from Assumption \ref{assu-kl}(ii) and Theorems~23.8 and 23.9 of \cite{Rockafellar-1970}, and the second equality follows from \cite[Theorem~6.6]{Rockafellar-1970}.
Combining the above display with \eqref{prop-kl-4}, we conclude that there exists
\begin{align}
\!\!{\bm u}^*\!=\!(u_1^*,\ldots, u_\ell^*)&\!\in\!{\rm ri}\,\mathcal{N}_{C_1}(A_1 x^*) \times \cdots\times  {\rm ri}\, \mathcal{N}_{C_\ell}(A_\ell x^*)\!=\!{\rm ri}\, \mathcal{N}_D(-{\bm A}\nabla f({\bm A}^T{\bm y}^*))\label{prop-kl-5}
\end{align}
such that
$
x^*-\bar{v}+{\bm A}^T {\bm u}^*=0
$.
This last equation together with \eqref{prop-kl-1} implies that ${\bm u}^*-{\bm y}^*\in {\rm ker}\, {\bm A}^T$. Consequently,
\begin{align}\label{prop-kl-6}
{\bm u}^*={\bm y}^*+({\bm u}^*-{\bm y}^*) \in {\bm y}^* + \ker {\bm A}^T.
\end{align}
Using \eqref{prop-kl-3}, \eqref{prop-kl-5} and \eqref{prop-kl-6}, we see that ${\bm u}^*\!\in\!({\bm A^T})^{-1}\!\{\!{\bm A}^T\!{\bm y}^*\}\cap{\rm ri}\,\mathcal{N}_D(-{\bm A}\nabla f({\bm A}^T{\bm y}^*))$, which implies that
\[
({\bm A^T})^{-1}\{{\bm A}^T{\bm y}^*\}\cap{\rm ri}\,\mathcal{N}_D(-{\bm A}\nabla f({\bm A}^T{\bm y}^*))\neq \emptyset.
\]
This and the polyhedrality of $({\bm A^T})^{-1}\{{\bm A}^T{\bm y}^*\}$ give the  bounded linear regularity for \eqref{prop-kl-2} at ${\bm y}^*$ (see \cite[Corollary 3]{Bauschke-Borwein-Li-1999}).
We can then deduce from Assumption~\ref{assu-kl}(i) and Theorem~\ref{thm-kl-g} that $d$ satisfies the KL property with exponent $\frac{1}{\alpha+1}$ at ${\bm y}^*$.
The desired conclusion now follows from the arbitrariness of ${\bm y^*}\in \Argmin d$ and \cite[Lemma~2.1]{Li-Pong-2018}.
\end{proof}

The next theorem establishes the uniformized growth condition and KL property for the function $d$, which is useful in deriving the Luo-Tseng type error bound in the next section.
We remark that existing results on uniformized KL property in \cite[Lemma 6]{Bolte-Sabach-Teboulle-2014} concern the KL property over a \emph{compact} set.
Since the set $\Argmin d$  can be unbounded in general, these existing results cannot be applied directly. Instead, we establish the uniformity result by noting that $d$ stays constant when moving along $E^\perp$. This is formally registered in the following auxiliary lemma, which is an immediate consequence of Corollary~2.5.5 and Theorem~2.5.3 of \cite{Auslender-Alfred-Teboulle-2006} because $d = q^*$ (see Proposition~\ref{prop-awb}).
\begin{lemma}\label{lemma-constant}
Consider \eqref{problem}. Let the function $d$ be given in \eqref{problem-duald} and $E$ be given in \eqref{prop-awb-3}. Suppose that Assumption \ref{assu-kl}(ii) holds.
Then for any ${\bm y}\in {\rm dom} \, d$, it holds that $d({\bm y} + {\bm u})=d({\bm y})$ whenever ${\bm u}\in E^\perp$.
\end{lemma}

\begin{theorem}[Uniformized KL property and growth condition]\label{the-ukl}
Consider \eqref{problem}.
Let the function $d$ be given in \eqref{problem-duald}.
Suppose that Assumption \ref{assu-kl} holds.
Then there exist positive constants $\epsilon$ and $c$ such that
\begin{align}
 &{\rm dist}({\bm y}, \Argmin d )\leq c\,(d({\bm y})-d^*)^{1-\frac{1}{1+\alpha}}\label{the-ukl-1}\\
 {and}\qquad&(d({\bm y})-d^*)^{\frac{1}{1+\alpha}}\leq  c\,{\rm dist}({\bm 0}, \partial d({\bm y}))\label{the-ukl-2}
\end{align}
whenever ${\bm y}\in Y_\epsilon:=\{{\bm y} : {\rm dist}({\bm y},  \Argmin d)\leq \epsilon, d^*\leq d({\bm y})\leq d^*+\epsilon\}$.
\end{theorem}
\begin{proof}
Under the assumptions, by Proposition \ref{prop-kl}, we know that $d$ is a KL function with exponent $\frac{1}{1+\alpha}$.  In \cite[Lemma 2.1(a)]{Auslender-Cominetti-Crouziex-1993}, it was shown that
\begin{align*}
\text{${\rm dom}\, q_E={\rm dom}\, q+E^\perp$ with ${\rm int}({\rm dom}\, q_E)={\rm ri}\,{\rm dom}\, q+E^\perp$},
\end{align*}
{\color{blue}where $q_E$ and $q$ are defined in Proposition \ref{prop-awb}.}
Now, note that $\partial q_E({\bm 0})(\subseteq  \Argmin d)$ is nonempty and compact, since we have ${\bm 0}\in {\rm  int}\,({\rm dom}\, q_E)$ in view of the above display and the fact that ${\bm 0}\in {\rm ri}\,{\rm dom}\, q$ (see Proposition \ref{prop-awb}).
Then, using \cite[Lemma 6]{Bolte-Sabach-Teboulle-2014} and \cite[Theorem 5]{Bolte-Nguyen-Peypouquet-Suter-2017}, there exist positive constants $\epsilon$ and $c$ so that
\begin{align}\label{the-ukl-3}
{\rm dist}({\bm z}, \Argmin d)\leq c (d({\bm z})-d^*)^{1-\frac{1}{1+\alpha}},
\end{align}
whenever ${\bm z}\in \bar{Y}_\epsilon:=\{ {\bm z}: {\rm dist}({\bm z}, \partial q_E({\bm 0}))\leq \epsilon, d^*\leq d({\bm z})\leq d^*+\epsilon\}$. Next, we show  that
\begin{align}\label{the-ukl-4}
Y_\epsilon=\bar{Y}_\epsilon+E^\perp.
\end{align}
First, recall \eqref{prop-awb-2}. For any ${\bm y}\in Y_\epsilon$, let ${\rm Proj}_{\Argmin d}({\bm y})=\hat{{\bm z}}+\hat{{\bm u}}$ with  $\hat{{\bm z}}\in \partial q_E({\bm 0})$ and $\hat{{\bm u}}\in E^\perp$.
Then
\[
{\rm dist}({\bm y}-\hat{{\bm u}}, \partial q_E({\bm 0})) \le \|({\bm y}-\hat{{\bm u}})-\hat{{\bm z}}\| = \|{\bm y}-{\rm Proj}_{ \Argmin d}{\color{blue}({\bm y})}\|\leq \epsilon;
\]
moreover, we also have $d({\bm y}-\hat{{\bm u}})=d({\bm y})$ thanks to Lemma \ref{lemma-constant}. From these we conclude that ${\bm y}-\hat{{\bm u}}\in \bar{Y}_\epsilon$.
Then  ${\bm y}=({\bm y}-\hat{{\bm u}})+\hat{{\bm u}}\in \bar{Y}_\epsilon+E^\perp$, which means $Y_\epsilon\subseteq\bar{Y}_\epsilon+E^\perp$.
To prove the converse inclusion, take any ${\bm z}+{\bm u}\in \bar{Y}_\epsilon+E^\perp$ with ${\bm z}\in\bar{Y}_\epsilon$ and ${\bm u}\in E^\perp$. Then it holds that
$d({\bm z}+{\bm u})=d({\bm z})$  thanks to Lemma \ref{lemma-constant} and
 \[
 {\rm dist}({\bm z}+{\bm u}, \Argmin d)\leq {\rm dist}({\bm z}+{\bm u},\partial q_E({\bm 0})+{\bm u})= {\rm dist}({\bm z},\partial q_E({\bm 0}))\leq \epsilon,
 \]
 where the first inequality follows from $\partial q_E({\bm 0})+{\bm u}\subseteq \Argmin d$ (see \eqref{prop-awb-2}). Then, $\bar{Y}_\epsilon+E^\perp\subseteq Y_\epsilon$ and
 relation \eqref{the-ukl-4} holds.

Now, for any ${\bm y}\in Y_\epsilon$, according to \eqref{the-ukl-4}, there exist ${\bm z}\in \bar{Y}_\epsilon$ and ${\bm u}\in E^{\perp}$   such that ${\bm y}={\bm z}+{\bm u}$.
 It then follows  that for any such ${\bm y}$,
\begin{align}
{\rm dist}({\bm y},  \Argmin d)&\overset{\rm (a)}={\rm dist}({\bm z}+{\bm u},  (\Argmin d)+{\bm u}) = {\rm dist}({\bm z}, \Argmin d)  \notag\\
&\overset{\rm (b)}\leq c (d({\bm z})-d^*)^{1-\frac{1}{1+\alpha}}\overset{\rm (c)}=c (d({\bm y})-d^*)^{1-\frac{1}{1+\alpha}}, \label{the-ukl-5}
\end{align}
where (a) follows from  $\Argmin d=(\Argmin d)+{\bm u}$ (which holds because $\Argmin d=\partial q_E({\bm 0})+E^\perp$ and ${\bm u}\in E^\perp$), (b) follows from \eqref{the-ukl-3}, and (c) follows from Lemma \ref{lemma-constant}.
This proves \eqref{the-ukl-1}. Finally, \eqref{the-ukl-1} together with \eqref{VA} implies \eqref{the-ukl-2}.
\end{proof}

Before ending this section, we demonstrate that the exponent in \eqref{the-ukl-1} is the ``best possible" under Assumption~\ref{assu-kl} by presenting an instance of BA-MSF \eqref{problem} with a path along which both sides of \eqref{the-ukl-1} vanish in the same order of magnitude.
{\color{blue}The following example which argues the tightness of the exponent is in line with the recent research on the study of tight error bounds: see the notions of exact modulus of the generalized concave KL property in \cite[Definition~6]{Wang-Wang-2022} and consistent error bound in \cite[Definition~3.1]{Liu-Lourenco-2022}.}
\begin{example}[Tightness of the exponent in \eqref{the-ukl-1}]\label{example-tight}
Consider 
\[
\min_{x\in\R^2}   {\color{blue}\frac{1}{2}}\|x-\bar{v}\|^2 \ \ {\rm s.t.}\ \  A_1 x\in C_1,
\]
where $\bar{v}=(2,0)$, $p\in(1,2]$,
$C_1=\{x\in \R^2: \|x\|_p\leq 1\}$, and
$A_1=\begin{bmatrix}
   1 & 0 \\
   0 & 0
 \end{bmatrix}$. Since $A_1^{-1}C_1 = \{x\in \R^2:\; A_1x \in C_1\} = [-1,1]\times \R$, it follows that $x^*=(1,0)$ is the unique solution.
Moreover, $C_1$ is $C^{1,\alpha}$-cone reducible with $\alpha = p-1$ in view of Example~\ref{c-1a-ball},
and it holds that $A_1^{-1}{\rm ri}\,C_1\neq \emptyset$ as $(0,0)\in ({\rm ri}\,C_1)\cap {\rm Range}(A_1)$.
Furthermore,
\begin{align*}
\begin{bmatrix}
  1 & 0
\end{bmatrix}^T=\bar{v}-x^*\in \big\{\begin{bmatrix}
  t & 0
\end{bmatrix}^T : t> 0\big\}={\rm ri}\, \mathcal{N}_{[-1,1]\times \R}(x^*) = {\rm ri}\, \mathcal{N}_{A_1^{-1} C_1}(x^*).
\end{align*}
Thus, Assumption \ref{assu-kl} is satisfied. Then Theorem~\ref{the-ukl} shows that \eqref{the-ukl-1} holds with exponent $1-\frac1{1+\alpha} = 1-\frac1p = \frac{p-1}{p}$, where the $d$ in \eqref{the-ukl-1} now takes the form
\[
d(y_1)= (1/2)\|A_1 y_1-\bar{v}\|^2- (1/2)\|\bar{v}\|^2+\|y_1\|_{\frac{p}{p-1}}.
\]
Next, in view of Proposition \ref{prop-primal-dual}, we have
$\begin{bmatrix}
  1 & 0
\end{bmatrix}^T=\bar{v}-x^*=A_1 \hat{y}_1$ whenever $\hat{y}_1 \in {\Argmin\, d}$. This implies that the first coordinate of $\hat y_1$ is $1$. Also, one can see from the definition of $d$ that the second coordinate of $\hat y_1$ is $0$. Thus, ${\Argmin\, d} = \{\begin{bmatrix}
  1 & 0
\end{bmatrix}^T\}$.

Now, let $y_1^\epsilon=(1,\epsilon)$ for $\epsilon \downarrow 0$. By direct computation, we obtain as $\epsilon\downarrow 0$ that
\begin{align*}
d(y_1^\epsilon)-d({\color{blue}\hat y_1})=(1+\epsilon^{\frac{p}{p-1}})^{\frac{p-1}{p}}-1=\Theta(\epsilon^{\frac{p}{p-1}})\quad{ and}\quad {\rm dist}\,(y_1^\epsilon, \Argmin d)=\Theta(\epsilon),
\end{align*}
showing that both sides of \eqref{the-ukl-1} vanish in the same order of magnitude along $y_1^\epsilon$.
\end{example}
{\color{blue}
\begin{remark}\label{remark-C1a}
When Assumption \ref{assu-kl} holds with $\alpha = 1$ in item (i) (the latter holds when each $C_i$ is $C^2$-cone reducible), Theorem \ref{the-ukl} asserts that the function $d$ given in \eqref{problem-duald} satisfies the quadratic growth condition; that is, \eqref{the-ukl-1} holds with exponent $1-\frac{1}{1+\alpha}=\frac{1}{2}$.
However, for the $p$-norm ball in Example \ref{c-1a-ball} with $p\in(1,2)$, in view of Example  \ref{example-tight}, $d$ satisfies \eqref{the-ukl-1} with a tight exponent $1-\frac{1}{\alpha+1}=\frac{p-1}{p}<\frac{1}{2}$.
Therefore, the $p$-norm ball in Example \ref{c-1a-ball} with $p\in(1,2)$ is $C^{1,\alpha}$-cone  reducible with $\alpha=p-1\in(0,1)$ but not $C^2$-cone reducible.
\end{remark}
}

\section{Convergence rate analysis}  \label{sec5}
In this section, we study the convergence rate of Algorithm~\ref{alg-Dykstra}. Recall from Theorem~\ref{the-dykstra-ori} that the (classical) Dykstra's projection algorithm is known to converge linearly when each $C_i$ is polyhedral. We will argue that the same conclusion holds for Algorithm~\ref{alg-Dykstra}. The proof technique will shed lights on how tools from Section~\ref{sec4} can be further developed to analyze convergence rate of Algorithm~\ref{alg-Dykstra} for a more general class of sets.

The analysis below relies on the following mapping $\mathcal{G}:\R^m\rightarrow \R^m$,
\begin{align}\label{G}
\mathcal{G}({\bm y}):={\bm y}-{\rm Prox}_{\sigma_D}({\bm y}-\nabla g({\bm y})),
\end{align}
where $g$ and  $D$ are given in \eqref{problem-duald} and \eqref{A-D}, respectively. Note that the ${\cal G}$ in \eqref{G} is instrumental in the framework of convergence rate analysis developed in \cite{Luo-Tseng-1993,Tseng-Yun-2009} for first-order methods.

\begin{theorem}[Linear convergence with polyhedral $C_i$]\label{the-rate-polyhedral}
Consider \eqref{problem} and let the function $d$ be given in \eqref{problem-duald}.
Suppose that each $C_i$ is polyhedral.  Then ${\Argmin d}\neq \emptyset$.
Moreover, if $\{x^t\}$, $\{{\bm y}^t\}$ are the sequences  generated by Algorithm \ref{alg-Dykstra}, then there exist  ${\bm y}^*\in  \Argmin d$, $r\in(0,1)$,  $a_0\in(0,1)$, $a_1>0$ and  positive integer $\bar{t}$ such that
\begin{align*}
\|x^t-x^*\| \leq a_1 a_0^t, \quad   d({\bm y}^{t+1})-d^*\leq r (d({\bm y}^{t})-d^*), \quad \|{\bm y}^t-{\bm y}^*\|\leq a_1 a_0^t, \quad \forall\, t\geq \bar{t},
\end{align*}
where $x^*$ is the unique solution of \eqref{problem}, and $d^*$ is given in \eqref{problem-duald}.
\end{theorem}
\begin{proof}
We first note from \cite[Proposition 8.29]{Rockafellar-Wets-2009} that $\sigma_{C_i}$ is a piecewise linear function for each $i$ and hence the function $d$ is a piecewise-linear-quadratic function.
Since $d$ is bounded below according to Proposition~\ref{prop-primal-dual}, invoking Frank-Wolfe theorem (see \cite[Theorem 2.8.1]{Cottle-Pang-Stone-1992}),
one can see that $ \Argmin d\neq \emptyset$. Combining this with \cite[Theorem 4]{Tseng-Yun-2009},
we assert that the following first-order error bound condition\footnote{This condition is also known as the Luo-Tseng error bound in the literature;
 see \cite{Luo-Tseng-1993, Tseng-Yun-2009, Li-Pong-2018, Yue-Zhou-So-2019}.} is satisfied: there exist positive constants $\epsilon_0$ and $c_0$ such that
\begin{equation}\label{LTeb}
{\rm dist}({\bm y}, \Argmin d)\leq c_0 \|\mathcal{G}({\bm y})\|\quad  \mbox{whenever $\|\mathcal{G}({\bm y})\|\leq \epsilon_0$ and $d({\bm y})\leq d^* + \epsilon_0$,}
\end{equation}
where ${\cal G}$ is given in \eqref{G}.
Then, the linear convergence results of $\{d({\bm y}^{t+1})\}$ and $\{{\bm y}^t\}$
follow from \cite[Theorem 2(b)]{Tseng-Yun-2009}.\footnote{One has to argue that the proximal CGD \eqref{initial-point} and \eqref{coordinate} is a special instance of the algorithm studied in \cite{Tseng-Yun-2009} so that \cite[Theorem 2(b)]{Tseng-Yun-2009} is applicable. In details, we see from Proposition~\ref{prop-descent}(i) that if one chooses $\sigma = 0.5$, $\gamma = 0$ and $\alpha^k_{\rm init} \equiv 1$ in the Armijo rule used in \cite{Tseng-Yun-2009} and consider the search direction $y_i^{t+1} - y_i^t$, then the Armijo rule will be satisfied with stepsize 1 so that $y_i^{t+1}$ will appear in the next iterate. Thus, the proximal CGD \eqref{initial-point} and \eqref{coordinate} coincides with the algorithm considered in \cite{Tseng-Yun-2009} with the aforementioned choices of $\sigma$, $\gamma$ and $\{\alpha^k_{\rm init}\}$.}
Moreover, from \eqref{xtdefinition} and Proposition~\ref{prop-primal-dual}, we have
 \begin{align}\label{the-rate-polyhedral-1}
\textstyle\|x^t-x^*\|=\left\|\bar{v}-\sum_{j=1}^\ell A_j^T y^t_j-\left(\bar{v}-\sum_{j=1}^\ell A_j^T y^*_j\right)\right\|\leq \|{\bm A}^T\| \|{\bm y}^t-{\bm y}^*\|.
\end{align}
The claimed convergence result of $\{x^t\}$ then follows immediately.
\end{proof}

Notice that in the above proof, the crucial ingredient is the first-order error bound condition \eqref{LTeb}, and it is known in \cite[Section~3]{Dmitriy-Lewis-2018} that this condition is intrinsically related to second-order growth condition. Below, leveraging the study of growth conditions in Section~\ref{sec4} for $C^{1,\alpha}$-cone reducible sets, we will develop an analogue of \eqref{LTeb} (with a general exponent on $\|{\cal G}({\bm y})\|$) to analyze the convergence rate of Algorithm~\ref{alg-Dykstra} when each $C_i$ is $C^{1,\alpha}$-cone reducible for some $\alpha\in (0,1]$.

\subsection{Convergence rate analysis for $C^{1,\alpha}$-cone reducible $C_i$}

We first derive a Luo-Tseng type error bound based on the study of growth conditions in Section~\ref{sec4}.
\begin{lemma}[Luo-Tseng type error bound]\label{lemma-Luo-Tseng}
Consider \eqref{problem}. Let $d^*$ and  the function $d$ be given in \eqref{problem-duald} and $\mathcal{G}$ be defined in \eqref{G}.
Suppose that Assumption \ref{assu-kl} holds. Then there exist $c > 0$ and $\epsilon>0$ such  that
\begin{align*}
{\rm dist}({\bm y}, \Argmin d)\leq c\|\mathcal{G}({\bm y})\|^\alpha \quad\text{whenever $\|{\cal G}(\bm y)\|\leq \epsilon$, $d^*\leq d({\bm y})\leq d^*+\epsilon$}.
\end{align*}
\end{lemma}
\begin{proof}
By Theorem~\ref{the-ukl}, there exist $\epsilon\in (0,{\color{blue}1/2})$ and $c > 0$ such that  for any  ${\bm y}\in Y_\epsilon$
\[
{\rm dist}({\bm y}, \Argmin d)\leq c(d({\bm y})-d^*)^{1-\frac{1}{1+\alpha}}\leq c\,{\rm dist}({\bm 0},\partial d({\bm y}))^{\frac{\alpha}{1+\alpha}}{\rm dist}({\bm y}, \Argmin d)^{\frac{\alpha}{1+\alpha}},
\]
where the second inequality follows from \eqref{VA}.
This further implies that
\begin{align}\label{lemma-Luo-Tseng-2}
{\rm dist}({\bm y}, \Argmin d)\leq c_1{\rm dist}({\bm 0},\partial d({\bm y}))^{\alpha}\ \ \mbox{for all }{\bm y}\in Y_\epsilon
\end{align}
where $c_1 := c^{1+\alpha}$.
Next, for any ${\bm y}\in Y_\epsilon$, we have\vspace{-0.2 cm}
\begin{equation*}
\left\{
\begin{aligned}
d^*\leq d({\rm prox}_d({\bm y})) &\le d({\rm prox}_d({\bm y})) + {\color{blue}\frac{1}{2}}\|{\bm y} - {\rm prox}_d({\bm y})\|^2 \\
&= \inf\limits_{\bm u\in \R^m} \{d({\bm u}) + {\color{blue}\frac{1}{2}}\|{\bm y} - {\bm u}\|^2\}\leq d({\bm y})\leq d^*+\epsilon,\\
{\rm dist}({\rm prox}_d({\bm y}),  \Argmin d)&\le \|{\rm prox}_d({\bm y})-\bar{\bm y}\| = \|{\rm prox}_d({\bm y})\!-\!{\rm prox}_d(\bar{\bm y})\|\\
&\overset{\rm (a)}\leq \|{\bm y}-\bar{\bm y}\|= {\rm dist}({\bm y}, \Argmin d) \leq \epsilon,
\end{aligned}
\right.
\end{equation*}
where $\bar{\bm y}:={\rm Proj}_{\Argmin d}({\bm y})$, (a) holds because the proximal operator is nonexpansive.
The above display implies that ${\rm prox}_d({\bm y})\in Y_\epsilon$.
Similarly, for any ${\bm y}\in Y_\epsilon$, it holds that
\begin{align}\label{lemma-Luo-Tseng-3}
\begin{split}
\|{\bm y}-{\rm prox}_d({\bm y})\|&\leq \|{\bm y}-\bar{\bm y}\|+\|{\rm prox}_d({\bm y})-\bar{\bm y}\| \leq 2\|{\bm y}-\bar{\bm y}\|\le 2\epsilon < 1.
\end{split}
\end{align}
Then we deduce for any ${\bm y}\in Y_\epsilon$ that
\begin{align*}
&{\rm dist}({\bm y}, \Argmin d)\leq \|{\bm y}-{\rm prox}_d({\bm y})\|+{\rm dist}({\rm prox}_d({\bm y}), \Argmin d) \notag \\
& \!\overset{{\rm (a)}}\leq\!  \|{\bm y}-{\rm prox}_d({\bm y})\|+ c_1{\rm dist}({\bm 0},\partial d({\rm prox}_d({\bm y})))^{\alpha}  \!\overset{{\rm (b)}}\leq\!  \|{\bm y}-{\rm prox}_d({\bm y})\|+ c_1\|{\bm y}-{\rm prox}_d({\bm y})\|^{\alpha}  \notag\\
&\overset{{\rm (c)}}\leq (1+c_1)\|{\bm y}-{\rm prox}_d({\bm y})\|^{\alpha} \overset{{\rm (d)}}\leq (1+c_1) c_2 \|\mathcal{G}({\bm y})\|^\alpha,  
\end{align*}
where (a) follows from \eqref{lemma-Luo-Tseng-2} and the fact that ${\rm prox}_d({\bm y})\in Y_\epsilon$, (b) follows from the fact that ${\bm y}-{\rm prox}_d({\bm y})\in \partial d({\rm prox}_d({\bm y}))$,
(c) holds as $\|{\bm y}-{\rm prox}_d({\bm y})\|< 1$ (thanks to \eqref{lemma-Luo-Tseng-3}) and $\alpha\in (0,1]$,
and (d) holds for some $c_2>0$ according to \cite[Theorem 3.5]{Dmitriy-Lewis-2018}.

We can now conclude that there exists $c_3 > 0$ such that ${\rm dist}({\bm y}, \Argmin d)\leq c_3\|\mathcal{G}({\bm y})\|^\alpha$ whenever ${\bm y}\in Y_\epsilon$. To complete the proof, it suffices to show that there exists $\nu \in (0,\epsilon)$ such that
 \[
 \{ {\bm y} : \|\mathcal{G}({\bm y})\|\leq \nu\}\subseteq \{ {\bm y} : {\rm dist}({\bm y}, \Argmin d)\leq \epsilon\}.
 \]
Suppose to the contrary that this is not true. Then there exists $\{{\bm z}^t\}$ such that ${\rm dist}({\bm z}^t, \Argmin d)> \epsilon$ for all $t$ and $\mathcal{G}({\bm z}^t)\to {\bm 0}$.
In view of the latter limit and \cite[Theorem 3.5]{Dmitriy-Lewis-2018}, we see that $\|{\bm z}^t - {\rm prox}_d({\bm z}^t)\|\to 0$.
Since ${\bm z}^t-{\rm prox}_d({\bm z}^t)\in \partial d({\rm prox}_d({\bm z}^t))$, we then deduce further that ${\rm dist}({\bm 0},\partial d({\rm prox}_d({\bm z}^t)))\to 0$,
and consequently ${\rm dist}({\rm prox}_d({\bm z}^t), \Argmin d)\to 0$ in view of Proposition~\ref{prop-awb}. We are now led to the following contradiction:
\[
\epsilon< {\rm dist}({\bm z}^t, \Argmin d)\le \|{\bm z}^t-{\rm prox}_d({\bm z}^t)\|+{\rm dist}({\rm prox}_d({\bm z}^t), \Argmin d)\to 0.
\]
This completes the proof.
\end{proof}

The following theorem establishes the convergence rate of  Algorithm \ref{alg-Dykstra}. The steps for deriving upper bounds on $\|{\cal G}({\bm y}^t)\|$ (see \eqref{the-rate-1-4}) and $d({\bm y}^{t+1}) - d^*$ (see \eqref{the-rate-2}) follow a similar argument as in
the proof of \cite[Theorem 2]{Tseng-Yun-2009}.
\begin{theorem}[convergence rate with $C^{1,\alpha}$-cone reducible $C_i$]\label{the-rate}
Consider \eqref{problem} and suppose that Assumption \ref{assu-kl} holds. Let $d$ be the function given in \eqref{problem-duald} and $\{x^t\}$, $\{{\bm y}^t\}$ be  the sequences  generated by Algorithm \ref{alg-Dykstra}.
Then it holds that $x^t\rightarrow x^*$ and ${\rm dist}({\bm y}^t, \Argmin d)\to 0$, where $x^*$ is the unique solution of \eqref{problem}. {\color{blue}Moreover, the following statements hold:}
\begin{enumerate}[{\rm (i)}]
  \item If each $C_i$ is $C^{1,1}$-cone reducible, then there exist ${\bm y}^*\in  \Argmin d$, {\color{blue}$a_1>0$, $a_0\in(0,1)$ and a positive integer $\bar{t}$} such that for any $t\geq \bar{t}$,
\begin{align*}
\|x^t-x^*\|\leq a_1 a_0^t,\quad \|{\bm y}^t-{\bm y}^*\|\leq a_1 a_0^t.
\end{align*}
  \item If each  $C_i$ is $C^{1,\alpha}$-cone reducible with $\alpha\in(0,1)$, then {\color{blue}there exist $a_1>0$ and a positive integer $\bar{t}$ such that}  for any $t\geq \bar{t}$,
{\color{blue}
\begin{align*}
&\|x^t-x^*\|\leq a_1 t^{-\frac{\alpha}{2}\cdot\frac{1-\theta}{2\theta-1}},\\
 & d({\bm y}^t)-d^*\leq a_1 t^{-\frac{1-\theta}{2\theta-1}},\\
  & {\rm dist}({\bm y}^t,\Argmin d)\leq a_1 t^{-\frac{\alpha}{2}\cdot\frac{1-\theta}{2\theta-1}},
\end{align*}
}
where $d^*$ is given in \eqref{problem-duald} {\color{blue}and $\theta:=\frac{1}{1+\alpha}\in(\frac{1}{2},1)$ is the KL exponent of $d$.}
\end{enumerate}
\end{theorem}
\begin{proof}
The conclusion that ${\rm dist}({\bm y}^t, \Argmin d)\to 0$ follows immediately from Theorem~\ref{the-convergence-awb}.
Recall from \eqref{A-D} that ${\bm A}^T=[A_1^T\ \cdots\ A_\ell^T]$.
Then the convergence of $\{x^t\}$ to $x^*$ can be deduced by noticing
\begin{align} \label{the-rate-0-1}
\textstyle \|x^t-x^*\|\!\overset{\rm (a)}=\!\left\|\bar{v}-\sum_{j=1}^\ell A_j^T y^t_j\!-\!\left(\bar{v}-\sum_{j=1}^\ell A_j^T \bar{y}^t_j\right)\right\|\!\leq\! \|{\bm A}^T\|{\rm dist}({\bm y}^{t}, \Argmin d),
\end{align}
where ${\rm Proj}_{ \Argmin d}({\bm y}^t)=: (\bar{y}^t_1,\ldots, \bar{y}^t_\ell)$ and (a) follows from Proposition \ref{prop-primal-dual} and \eqref{xtdefinition}.

Recall from \eqref{A-D} that  $D=C_1\times \cdots \times C_\ell$.
Define
\begin{align*}
\mathcal{G}_i^\gamma({\bm y}):=y_i-{\rm Prox}_{\gamma^{-1}\sigma_{C_i}}(y_i-\gamma^{-1}\nabla_{y_i} g({\bm y})), \quad i=1,\ldots, \ell,
\end{align*}
where $g$ is as in \eqref{problem-duald}.
Then we can write $\mathcal{G}({\bm y})=(\mathcal{G}^1_1({\bm y}), \ldots, \mathcal{G}^1_\ell({\bm y}))$ and
\eqref{coordinate} becomes
\begin{align}\label{the-rate-0}
y^{t+1}_i={\rm Prox}_{\gamma_i^{-1}\sigma_{C_i}}(y^t_i-\gamma_i^{-1}\nabla_{y_i} g(\tilde{{\bm y}}^{t+1}_{i-1})) = y^t_i-\mathcal{G}_i^{\gamma_i}(\tilde{{\bm y}}^{t+1}_{i-1}),
\end{align}
where  $\tilde{{\bm y}}^{t+1}_{i-1}$ is as in \eqref{y-tilde}.
Let $\hat{y}^{t+1}_i:=y^t_i-\mathcal{G}_i^1(\tilde{{\bm y}}^{t+1}_{i-1}) = {\rm Prox}_{\sigma_{C_i}}(y^t_i-\nabla_{y_i} g(\tilde{{\bm y}}^{t+1}_{i-1}))$.
By the definition of the proximal operator, we see that $\hat{y}^{t+1}_i$ and  $y^{t+1}_i$  satisfy  the respective first-order optimality conditions:
\begin{align*}
&0 \in \nabla_{y_i}g(\tilde{{\bm y}}^{t+1}_{i-1})+ (\hat{y}^{t+1}_i-y^{t}_i)+\partial \sigma_{C_i}(\hat{y}^{t+1}_i), \\
& 0 \in \nabla_{y_i}g(\tilde{{\bm y}}^{t+1}_{i-1})+\gamma_i (y^{t+1}_i-y^{t}_i)+\partial \sigma_{C_i}(y^{t+1}_i),
\end{align*}
 which implies that
\begin{align*}
&\hat{y}^{t+1}_i\in \Argmin_{z\in \R^{m_i}} \langle\nabla_{y_i}g(\tilde{{\bm y}}^{t+1}_{i-1})+ (\hat{y}^{t+1}_i-y^{t}_i),z\rangle+ \sigma_{C_i}(z),\\
&{y}^{t+1}_i\in \Argmin_{z\in \R^{m_i}} \langle\nabla_{y_i}g(\tilde{{\bm y}}^{t+1}_{i-1})+\gamma_i ({y}^{t+1}_i-y^{t}_i),z\rangle+ \sigma_{C_i}(z).
\end{align*}
Then, it holds that for any $z\in \R^{m_i}$
\begin{align}
&\langle\nabla_{y_i}g(\tilde{{\bm y}}^{t+1}_{i-1})+(\hat{y}^{t+1}_i-y^t_i),\hat{y}^{t+1}_i-z\rangle+\sigma_{C_i}(\hat{y}^{t+1}_i)-\sigma_{C_i}(z)\leq 0, \label{the-rate-1}\\
&\langle\nabla_{y_i}g(\tilde{{\bm y}}^{t+1}_{i-1})+\gamma_i(y^{t+1}_i-y^t_i),y^{t+1}_i-z\rangle+\sigma_{C_i}(y^{t+1}_i)-\sigma_{C_i}(z)\leq 0. \label{the-rate-1-1}
\end{align}
Substituting $z=y^{t+1}_i$ and $z=\hat{y}^{t+1}_i$ into \eqref{the-rate-1} and \eqref{the-rate-1-1} respectively and summing them together, we obtain upon recalling \eqref{the-rate-0} and the fact $\hat{y}^{t+1}_i-y^t_i= -\mathcal{G}_i^1(\tilde{{\bm y}}^{t+1}_{i-1})$ that
\[
\langle\mathcal{G}_i^1(\tilde{{\bm y}}^{t+1}_{i-1}),\mathcal{G}_i^1(\tilde{{\bm y}}^{t+1}_{i-1})-\mathcal{G}^{\gamma_i}_i(\tilde{{\bm y}}^{t+1}_{i-1})\rangle
+\gamma_i\langle\mathcal{G}_i^{\gamma_i}(\tilde{{\bm y}}^{t+1}_{i-1}),\mathcal{G}_i^{\gamma_i}(\tilde{{\bm y}}^{t+1}_{i-1})-\mathcal{G}^{1}_i(\tilde{{\bm y}}^{t+1}_{i-1})\rangle\leq 0,
\]
which gives $\|\mathcal{G}_i^1(\tilde{{\bm y}}^{t+1}_{i-1})\|^2 + \gamma_i\|\mathcal{G}_i^{\gamma_i}(\tilde{{\bm y}}^{t+1}_{i-1})\|^2\le (1+\gamma_i) \langle\mathcal{G}_i^{\gamma_i}(\tilde{{\bm y}}^{t+1}_{i-1}),\mathcal{G}_i^1(\tilde{{\bm y}}^{t+1}_{i-1})\rangle$, and hence
\begin{align}\label{the-rate-1-2}
\|\mathcal{G}_i^1(\tilde{{\bm y}}^{t+1}_{i-1})\|\leq (1+\gamma_i)\|\mathcal{G}_i^{\gamma_i}(\tilde{{\bm y}}^{t+1}_{i-1})\|=(1+\gamma_i)\|y^{t+1}_i-y^t_i\|,
\end{align}
where the last equality follows from \eqref{the-rate-0}.
Recall that ${\bm y}^{t}=({y}^{t}_1,\ldots,y^{t}_\ell)$. Then 
\begin{align}
&\|\mathcal{G}_i^1({\bm y}^t)-\mathcal{G}_i^1(\tilde{\bm y}^{t+1}_{i-1})\|\notag\\
=&\| y^{t}_i-{\rm prox}_{\sigma_{C_i}}(y^{t}_i-\nabla_{y_i}g({\bm y}^{t}))-(\tilde{\bm y}^{t+1}_{i-1})_i+{\rm prox}_{\sigma_{C_i}}((\tilde{\bm y}^{t+1}_{i-1})_i-\nabla_{y_i}g(\tilde{\bm y}^{t+1}_{i-1}))   \| \notag\\
\leq & L_g \|{\bm y}^{t+1}-{\bm y}^t\|, \label{the-rate-1-3}
\end{align}
where the inequality follows from the fact that $(\tilde{\bm y}^{t+1}_{i-1})_i=y^t_i$, the nonexpansiveness of the proximal operator, and the Lipschitz continuity of $\nabla g$ with Lipschitz modulus $L_g:=\|{\bm A}^T{\bm A}\|$.
We then deduce further that
\begin{align}
\|\mathcal{G}({\bm y}^t)\|&\textstyle\leq \sum_{i=1}^\ell \|\mathcal{G}^1_i({\bm y}^t)\|\leq \sum_{i=1}^\ell \left(\|\mathcal{G}^1_i({\bm y}^t)-\mathcal{G}^1_i(\tilde{\bm y}^{t+1}_{i-1})\|+\|\mathcal{G}^1_i(\tilde{\bm y}^{t+1}_{i-1})\|\right)\notag\\
&\textstyle\overset{\rm (a)}\leq \sum_{i=1}^\ell \left(L_g\|{\bm y}^{t+1}-{\bm y}^t\|+(1+\gamma_i)\|y^{t+1}_i-y^t_i\| \right) \le \hat M\|{\bm y}^{t+1}-{\bm y}^t\|,   \label{the-rate-1-4}
\end{align}
where (a) follows from \eqref{the-rate-1-2} and  \eqref{the-rate-1-3}, and $\hat M := \ell L_g + (\ell+\sum_{i=1}^\ell\gamma_i)\sqrt{\ell}$.

Next, let $\bar{{\bm y}}^t=(\bar{y}^t_1,\ldots, \bar{y}^t_\ell):={\rm Proj}_{ \Argmin d}({\bm y}^t)$. Then we see that
\begin{align}
&d({\bm y}^{t+1})-d^* \overset{{\rm (a)}}= g({\bm y}^{t+1})+\sigma_{D}({\bm y}^{t+1})-g(\bar{{\bm y}}^t)-\sigma_{D}(\bar{{\bm y}}^t)\notag \\
\overset{{\rm (b)}}=&\textstyle\langle\nabla g({\bm u}^t),{\bm y}^{t+1}-\bar{{\bm y}}^t\rangle +\sum_{i=1}^\ell [\sigma_{C_i}({y}^{t+1}_i)-\sigma_{C_i}(\bar{y}^{t}_i)]\notag\\
=&\textstyle\langle\nabla g({\bm u}^t)-\nabla g({\bm y}^t),{\bm y}^{t+1}-\bar{{\bm y}}^t\rangle+ \sum_{i=1}^\ell \langle\nabla_{y_i} g({\bm y}^t)-\nabla_{y_i}g(\tilde{{\bm y}}^{t+1}_{i-1}),{y}^{t+1}_i-\bar{y}_i^t\rangle\notag\\
&\textstyle+\sum_{i=1}^\ell [\langle\nabla_{y_i}g(\tilde{{\bm y}}^{t+1}_{i-1})+\gamma_i({y}^{t+1}_i-{y}^{t}_i),{y}^{t+1}_i-\bar{y}^{t}_i\rangle+\sigma_{C_i}({y}^{t+1}_i)-\sigma_{C_i}(\bar{y}^{t}_i)]\notag\\
&\textstyle-\sum_{i=1}^\ell \gamma_i \langle{y}^{t+1}_i-{y}^{t}_i,{y}^{t+1}_i-\bar{y}^{t}_i\rangle\notag \\
\overset{{\rm (c)}}\leq& L_g\|{\bm u}^t-{\bm y}^t\|\|{\bm y}^{t+1}-\bar{{\bm y}}^t\|+{\ell L_g}\|{\bm y}^{t+1}-{\bm y}^t\| \|{\bm y}^{t+1}-\bar{{\bm y}}^t\|\notag\\
&+\gamma_{\rm sum}\|{\bm y}^{t+1}-{\bm y}^t\| \|{\bm y}^{t+1}-\bar{{\bm y}}^t\| \notag\\
\leq & [(L_g+{\ell L_g}+\gamma_{\rm sum}) \|{\bm y}^{t+1}-{\bm y}^t\|+L_g\|{\bm y}^{t}-\bar{{\bm y}}^t\|](\|{\bm y}^{t+1}-{\bm y}^t\|+\|{\bm y}^t-\bar{{\bm y}}^t\|), \label{the-rate-2}
\end{align}
where (a) follows from the definition of $d$ in \eqref{problem-duald}, (b) follows from the mean value theorem with ${\bm u}^t=\tau_t {\bm y}^{t+1}+(1-\tau_t)\bar{{\bm y}}^t$ for some $\tau_t\in(0,1)$,
(c) follows from \eqref{the-rate-1-1} and the Lipschitz continuity of $\nabla g$ with $\gamma_{\rm sum}:=\sum_{i=1}^\ell\gamma_i$.

On the other hand, we have $\|\mathcal{G}({\bm y}^t)\|\to 0$  (thanks to \eqref{the-rate-1-4} and Proposition \ref{prop-descent}(iii)) and     $d({\bm y}^t)\rightarrow d^*$ (thanks to Theorem~\ref{the-convergence-awb}).
Thus, there exists sufficiently large positive integer $\bar{t}$ such that
$\|\mathcal{G}({\bm y}^t)\|\leq {\epsilon}$ and $d({\bm y}^t)\leq d^*+{\epsilon}$ for any $t\geq \bar{t}$, where ${\epsilon}$ is specified in Lemma~\ref{lemma-Luo-Tseng}. Then, it holds that for any $t\ge \bar{t}$,
\begin{align}\label{the-rate-3}
\|{\bm y}^t-\bar{{\bm y}}^t\|\overset{\rm (a)}={\rm dist}({\bm y}^t, \Argmin d)\overset{\rm (b)}\leq {c}\|\mathcal{G}({\bm y}^t)\|^\alpha\overset{\rm (c)}\leq \tilde{c}\|{\bm y}^{t+1}-{\bm y}^t\|^\alpha,
\end{align}
where (a) holds as $\bar{{\bm y}}^t={\rm Proj}_{ \Argmin d}({\bm y}^t)$, (b) follows from Lemma~\ref{lemma-Luo-Tseng},   (c) holds for  $\tilde{c}:=c\hat M^\alpha$ thanks to \eqref{the-rate-1-4}.
Now, since $\|{\bm y}^{t+1}-{\bm y}^t\|\to 0$ (see Proposition \ref{prop-descent}(iii)) and $\alpha\in(0,1]$, we deduce from \eqref{the-rate-2} and \eqref{the-rate-3} that there exists $c_1>0$ such that
\begin{align}\label{the-rate-4}
d({\bm y}^{t+1})-d^*\leq c_1\|{\bm y}^{t+1}-{\bm y}^t\|^{2\alpha} \ \ \mbox{for all }t\ge \bar t.
\end{align}
In addition, Proposition \ref{prop-descent}(ii) shows that for all $t$,
\begin{align}\label{the-rate-5}
d({\bm y}^{t+1})-d({\bm y}^t)\leq  -c_2 \|{\bm y}^{t+1}-{\bm y}^{t}\|^2,
\end{align}
for some $c_2>0$.
Then, combining the above two displays, we have for all  $t\geq \bar{t}$,
\[
(d({\bm y}^{t+1})-d^*)-(d({\bm y}^t)-d^*)\leq   -c_2 \|{\bm y}^{t+1}-{\bm y}^{t}\|^2 \leq -c_1^{-1/\alpha} c_2 (d({\bm y}^{t+1})-d^*)^{1/\alpha},
\]
which implies that
\begin{align}\label{the-rate-6}
(d({\bm y}^{t+1})-d^*)+c_1^{-1/\alpha} c_2 (d({\bm y}^{t+1})-d^*)^{1/\alpha}\leq d({\bm y}^t)-d^* \ \ \mbox{for all }t\ge \bar t.
\end{align}
Next, we consider the cases $\alpha=1$ and $\alpha\in(0,1)$ separately.

{\bf Case 1}: $\alpha=1$. In this case, from \eqref{the-rate-6} we have  for  $t\geq \bar{t}$
\[
\textstyle d({\bm y}^{t+1})-d^*\leq \frac{1}{1+c_1^{-1}c_2} (d({\bm y}^t)-d^*),
\]
which implies that $\{d({\bm y}^{t+1})-d^*\}$ is Q-linearly convergent to zero. The last inequality together with \eqref{the-rate-5} shows that
\[
\|{\bm y}^{t+1}-{\bm y}^{t}\|^2\leq c_2^{-1} [(d({\bm y}^t)-d^*)-(d({\bm y}^{t+1})-d^*)]\leq c^{-1}_2 (d({\bm y}^t)-d^*),
\]
which together with the Q-linearly convergence (to zero) of $\{d({\bm y}^{t})-d^*\}$ implies that there exists $c_3>0, \alpha_0\in(0,1)$ such that
$
\|{\bm y}^{t+1}-{\bm y}^{t}\|\leq c_3 \alpha_0^t
$.
Thus, we have
\[
\textstyle\|{\bm y}^{t_1}-{\bm y}^{t_2}\|\leq \sum_{j=t_1}^{t_2-1}\|{\bm y}^{j+1}-{\bm y}^{j}\|\leq \frac{c_3 }{1-\alpha_0}\alpha_0^{t_1}\ \ \mbox{for all }t_2>t_1\geq \bar{t},
\]
which means that $\{{\bm y}^t\}$ is convergent. Let ${\bm y}^*$ denote its limit.
Then ${\bm y}^*\in  \Argmin d$ by Theorem \ref{the-convergence-awb}.
Passing to the limit as $t_2\rightarrow \infty$ in the above display, we obtain
\[
\|{\bm y}^{t_1}-{\bm y}^*\|\leq (c_3 \alpha_0^{t_1})/(1-\alpha_0).
\]
Using this, \eqref{xtdefinition} and Proposition~\ref{prop-primal-dual}, we deduce further that
for $t\geq \bar{t}$
\begin{align}\label{the-rate-6-1}
\textstyle\|x^t-x^*\|=\left\|\bar{v}-\sum_{j=1}^\ell A_j^T y^t_j-\left(\bar{v}-\sum_{j=1}^\ell A_j^T y^*_j\right)\right\|\leq \|{\bm A}^T\|\frac{c_3 }{1-\alpha_0} \alpha_0^{t},
\end{align}
which implies the linear convergence of $\{x^t\}$ to $x^*$.

{\bf Case 2}: $\alpha\in(0,1)$. We first show that the sequence $\{d({\bm y}^t)-d^*\}$  is sublinearly convergent to zero.
{\color{blue}To this end, recalling that  $\theta=\frac{1}{\alpha+1}$,  \eqref{the-rate-6} can be written as
\begin{align}
 (d({\bm y}^{t+1})-d^*)^{\frac{\theta}{1-\theta}}\leq c_2^{-1}c_1^{\frac{\theta}{1-\theta}} [(d({\bm y}^t)-d^*)-(d({\bm y}^{t+1})-d^*)] \ \ \mbox{for all }t\ge \bar t.
\end{align}
Then, following the same arguments in \cite[Theorem 2]{Attouch-Bolte-2009} starting from \cite[Equation (13)]{Attouch-Bolte-2009}, we can deduce that there exists   $c_4>0$ such that
\begin{align}\label{the-rate-8}
 d({\bm y}^t)-d^*\leq c_4 t^{-\frac{1-\theta}{2\theta-1}}  \ \ \mbox{for all }t\ge \bar t;
\end{align}
where the last inequality corresponds to the first inequality on \cite[Page 15]{Attouch-Bolte-2009}.
}

For the sublinear convergence of $\{{\rm dist}({\bm y}^{t}, \Argmin d)\}$,
we see that for any $t\geq \bar{t}$
\begin{align}
c_2\tilde{c}^{-\frac{2}{\alpha}}  {\rm dist}({\bm y}^{t}, \Argmin d)^{\frac{2}{\alpha}}&\overset{{\rm (a)}}\leq c_2\tilde{c}^{-\frac{2}{\alpha}}(\tilde{c} \|{\bm y}^{t+1}-{\bm y}^{t}\|^\alpha)^{2/\alpha}=
 c_2 \|{\bm y}^{t+1}-{\bm y}^t\|^2 \notag\\
&\overset{{\rm (b)}}\leq  d({\bm y}^{t})-d({\bm y}^{t+1})\leq d({\bm y}^{t})-d^*,   \label{the-rate-9}
\end{align}
where (a)  follows from \eqref{the-rate-3} and (b) follows from \eqref{the-rate-5}.
Thus, the sublinear convergence rate of $\{{\rm dist}({\bm y}^{t}, \Argmin d)\}$ follows from \eqref{the-rate-8}.
Finally,  \eqref{the-rate-0-1} implies that the convergence rate of $\{x^t\}$ can be obtained from that of $\{{\rm dist}({\bm y}^{t}, \Argmin d)\}$.
\end{proof}
{\color{blue}
\begin{remark}
We reiterate that the convergence framework based on the KL property (see, e.g., \cite[Theorem 1]{Attouch-Bolte-2009}) is not directly applicable here because $\{{\bm y}^t\}$ may not have accumulation points. On the other hand, if $\{{\bm y}^t\}$ is bounded, then the convergence rate for $\{d({\bm y}^t) - d^*\}$ in Theorem~\ref{the-rate}(ii) is different by a factor of $1-\theta$ from the standard rate derived from the KL-based analysis; see, for example, \cite[Theorem~4.1(iv)]{Garr-Rosa-Villa-2023}.
\end{remark}
}

Theorem \ref{the-rate}(i) shows linear convergence rate of Algorithm \ref{alg-Dykstra} under a genericity assumption on $\bar{v}$, i.e., $0\in x^*-\bar{v}+ {\rm ri}\, \partial (\sum_{i=1}^\ell \delta_{A_i^{-1}C_i })(x^*)$.
To see that this condition is indispensable, we give an example that satisfies all  assumptions  in Theorem~\ref{the-rate}(i) except for Assumption~\ref{assu-kl}(iii), and linear convergence fails.

\begin{example}[Linear convergence fails]\label{example}
Consider the following problem
\begin{align}\label{example-1}
\textstyle \min\limits_{y_1,y_2\in\R^3} d(y_1,y_2):=  \textstyle \frac{1}{2}\|y_1+y_2-\bar{v}\|^2+\sigma_{C_1}(y_1)+\sigma_{C_2}(y_2),
\end{align}
where $\bar{v}=(1, -1,1)$ and
\[
C_1:=\{(x_1,x_2,x_3)\in \R^3: x_3\leq -\|(x_1,x_2)\|\},\ \ \
C_2:=\{(0,x_2,x_3)\in \R^3: x_2, x_3\in\R\}.
\]
 Note that \eqref{example-1} is the negative of the Lagrange dual (up to an additive constant) of the BA problem
$
\min_{x\in C_1\cap C_2} \frac12\|x - \bar {v}\|^2.
$
By the definition of $C_1$ and $C_2$, one can see that $C_1$ and $C_2$ are $C^{1,1}$-cone reducible (see also Example~\ref{c-1a-cone}) and ${\rm ri}\, C_1\cap{\rm ri}\, C_2\neq \emptyset$.

Notice that $C_1=K_1^\circ$ and $C_2 = K_2^\circ$, where
\begin{align*}
K_1:=\{(x_1,x_2,x_3)\in \R^3: x_3\geq \|(x_1,x_2)\|\},\quad  K_2:=\{ (x_1,0,0) \in \R^3: x_1\in\R\}.
\end{align*}
Then $\sigma_{C_1}=\delta_{K_1}$ and $\sigma_{C_2}=\delta_{K_2}$ according to \cite[Example 11.4(b)]{Rockafellar-Wets-2009}, and \eqref{example-1} can be equivalently written as
\begin{align}\label{mind}
\textstyle\min\limits_{y_1,y_2\in\R^3} d(y_1,y_2) =  \textstyle \frac{1}{2}\|y_1+y_2-\bar{v}\|^2+\delta_{K_1}(y_1)+\delta_{K_2}(y_2).
\end{align}
One can verify that the optimal value of the above problem is zero, and
 \[
 y_1^*=(0,-1,1)\in K_1 \ \ {\rm and}\ \  y_2^*=(1, 0,0)\in K_2
 \]
is the unique solution of \eqref{example-1}, i.e., $ \Argmin d=\{(y_1^*,y_2^*)\}$, and $x^*=(0,0,0)$ is the unique solution of the primal problem in view of Proposition~\ref{prop-primal-dual}. Moreover, we have
\[
0\notin x^*-\bar{v}+{\rm ri}\, \partial (\delta_{ C_1}+\delta_{ C_2})(x^*),
\]
since ${\rm ri}\, \partial (\delta_{ C_1}+\delta_{ C_2})(x^*)={\rm ri}\,\mathcal{N}_{C_1\cap C_2}(x^*)\subseteq \mathcal{N}_{C_1\cap C_2}(x^*)= \{(x_1, x_2,x_3) : x_3\geq |x_2|\}$.

Recall that the $y$-iterates in Algorithm~\ref{alg-Dykstra} can be obtained by applying CGD to \eqref{mind}, with the initial points $y_1^0=y_2^0=(0,0,0)$. By induction,
the $y$-iterates generated by Algorithm~\ref{alg-Dykstra} can be written as\footnote{The projection of $(u,t)\in \R^{n+1}$ onto $K_1$ can be found in \cite[Exercise~29.11]{Bauschke-Combettes-2017}. }
\begin{align}\label{example-2}
\begin{cases}
 & y_1^{t+1}={\rm Proj}_{K_1}(\bar v - y_2^t)
  = \left( a_{t+1},-\frac{1}{2}\left(1+\frac{1}{\sqrt{a_t^2+1}}\right), \frac{1}{2}\left(1+\sqrt{a_t^2+1}\right) \right),   \\
 & y_2^{t+1}={\rm Proj}_{K_2}(\bar v - y_1^{t+1})=(1-a_{t+1},0,0),   \\
 & \text{with $a_{t+1}=\frac{1}{2}\left( 1+\frac{1}{\sqrt{a_t^2+1}}\right)a_t$ and $a_0=1$}. 
\end{cases}
\end{align}
Observing from above that $\{(y_1^t,y_2^t)\}$  is convergent since $\{a_t\}$ is bounded and nonincreasing,
 we know that $(y_1^t, y_2^t)\rightarrow (y_1^*,y_2^*)$ by Proposition \ref{prop-descent}(v). This implies that $\lim_{t\rightarrow \infty}a_t= 0$. From the definition of $a_t$ in \eqref{example-2} and noting that $a_t>0$, we have
\begin{align*}
\frac{1}{a^2_{t+1}}\!-\!\frac{1}{a^2_{t}}\!=\!\frac{3a_t^2+2-2\sqrt{a_t^2+1}}{a_t^2(1+\sqrt{a_t^2+1})^2}\!=\!
\frac{3a_t^2+2-2(1+0.5a_t^2+O(a_t^4))}{a_t^2(1+\sqrt{a_t^2+1})^2}\!=\!\frac{2+O(a_t^2)}{(1+\sqrt{a_t^2+1})^2}.
\end{align*}
 Summing  both sides of the above equality from $t=0$ to $t=N$, we get
$\frac{1}{a^2_{N}}-\frac{1}{a^2_{0}}=\Theta(N),$
which implies that
$
a_{t}= \Theta(1/\sqrt{t})
$.
Then, using this relation, \eqref{example-2} and the fact $\sqrt{h^2+1}-1=\Theta(h^2)$ as $h\rightarrow 0$, we deduce that
${\rm dist}((y_1^{t+1},y_2^{t+1}),  \Argmin d)^2
=\|(y_1^{t+1},y_2^{t+1})-(y_1^*,y_2^*)\|^2$ is equal to
\begin{align*}
\textstyle 2a_{t+1}^2+\left(\frac{\sqrt{a_t^2+1}-1}{2\sqrt{a_t^2+1}}\right)^2+\left( \frac{\sqrt{a_t^2+1}-1}{2}\right)^2
=2a_{t+1}^2 +\Theta(a_t^4)=\Theta \left(\frac{1}{{t+1}}\right),
\end{align*}
which means that the convergence rate of $\{(y_1^t, y_2^t)\}$ to $(y_1^*, y_2^*)$ is not linear.
\end{example}

\appendix
\section{Equivalence between Algorithm \ref{alg-Dykstra} and a proximal CGD scheme}\label{app-derivation}
We first derive Algorithm \ref{alg-Dykstra} from \eqref{initial-point} and \eqref{coordinate}.
To this end, note that \eqref{coordinate} is equivalent to
\begin{align}
 y^{t+1}_i =\, {\rm prox}_{\gamma_i^{-1}\sigma_{C_i}}(y^t_i-\gamma_i^{-1}\nabla_{y_i} g(\tilde{{\bm y}}^{t+1}_{i-1})).  \label{app-derivation-2}
 \end{align}
Define $x^0_\ell := \bar v$ and, for $i = 0,1,\ldots,\ell$,
\begin{align}\label{app-derivation-2-1}
\textstyle x^{t+1}_{i}:=\bar{v}-\sum_{j=1}^i  A_j^T y^{t+1}_{j}-\sum_{j=i+1}^{\ell} A_j^T y^{t}_j.
\end{align}
Then according to \eqref{app-derivation-2-1} and the definition of $\tilde{{\bm y}}^{t+1}_{i-1}$ in \eqref{y-tilde}, it holds that for $i = 1,\ldots,\ell$,
\begin{align}
A_ix^{t+1}_{i-1}&=-\nabla_{y_{i}} g(\tilde{{\bm y}}^{t+1}_{i-1}), \label{app-derivation-3}\\
x^{t+1}_{i}&=x^{t+1}_{i-1}-A_i^T  (y^{t+1}_i-y^{t}_i).\label{app-derivation-4-2}
\end{align}
Consequently,
 relation \eqref{app-derivation-2} can be further rewritten as
\begin{align}
y^{t+1}_i&\overset{\rm (a)}=y^t_i-\gamma_i^{-1}\nabla_{y_i} g(\tilde{{\bm y}}^{t+1}_{i-1})-\gamma_i^{-1}{\rm Proj}_{C_i}(\gamma_i y^t_i-\nabla_{y_i} g(\tilde{{\bm y}}^{t+1}_{i-1})) \label{app-derivation-3.9}\\
&\overset{\rm  \eqref{app-derivation-3}}=y^t_i+\gamma_i^{-1}A_ix^{t+1}_{i-1}-\gamma_i^{-1}{\rm Proj}_{C_i}(\gamma_i y^t_i+ A_ix^{t+1}_{i-1}),\label{app-derivation-4}
\end{align}
where (a) holds since for any closed convex set $K\subseteq \R^n$, $r>0$ and $u\in \R^n$, we have
$
{\rm prox}_{r\sigma_K}(u)
=u-{\rm prox}_{(\sigma_{rK})^*}(u)=u-{\rm Proj}_{rK}(u)=u-r{\rm Proj}_{K}(r^{-1}u)
$.
Then, using \eqref{app-derivation-4-2} and the expression of $y^{t+1}_i - y^{t}_i$ derived from \eqref{app-derivation-4}, we have for $i=1,\ldots,\ell$,
\begin{align}
x^{t+1}_{i}&=(I-\gamma_i^{-1}A_i^T A_i )x^{t+1}_{i-1}+ \gamma_i^{-1}A_i^T {\rm Proj}_{C_i}(\gamma_i y^t_i+A_ix^{t+1}_{i-1}).\label{app-derivation-5}
\end{align}
Thus, we have shown that \eqref{coordinate} gives \eqref{app-derivation-4} and \eqref{app-derivation-5}, i.e., \eqref{alg-1} in Algorithm~\ref{alg-Dykstra}.
Also, from \eqref{initial-point}, \eqref{app-derivation-2-1} and the definition that $x^0_\ell := \bar v$, we see that
\begin{align}\label{app-derivation-5-1}
\textstyle x^{t+1}_0=\bar{v}-\sum_{j=1}^\ell A_j^T y_j^{t} = x^t_\ell\ \ \mbox{for all }t\geq 0.
\end{align}
Combining \eqref{app-derivation-4},  \eqref{app-derivation-5}  and  \eqref{app-derivation-5-1},  we derive  Algorithm \ref{alg-Dykstra}.

Conversely, we show that \eqref{initial-point} and \eqref{coordinate} can be deduced from Algorithm~\ref{alg-Dykstra}. Thanks to Step 1 of Algorithm~\ref{alg-Dykstra}, it suffices to prove \eqref{app-derivation-2}. To this end, let $\{x^t\}$, $\{x_i^t\}$ and $\{y_i^t\}$, $i = 1,\ldots,\ell$, be generated by Algorithm \ref{alg-Dykstra}. Then we obtain \eqref{app-derivation-4-2} from \eqref{alg-1}.
We claim that \eqref{xtdefinition} holds by induction. First, it clearly holds for $t = 0$ from Step 1 of Algorithm~\ref{alg-Dykstra}. Suppose that $x^t_\ell = \bar v - {\bm A}^T {\bm y}^t$ for some $t\ge 0$. Then the beginning of Step 2 of Algorithm~\ref{alg-Dykstra} shows that $x_0^{t+1} = \bar v - {\bm A}^T {\bm y}^t$, which together with \eqref{app-derivation-4-2} shows that $x^{t+1}_\ell = \bar v - {\bm A}^T {\bm y}^{t+1}$, thus establishes \eqref{xtdefinition} by induction.

Now, \eqref{xtdefinition}, \eqref{app-derivation-4-2} and $x_0^{t+1} = x^t_\ell$ (Step 2 of Algorithm~\ref{alg-Dykstra}) give \eqref{app-derivation-2-1}, and hence \eqref{app-derivation-3}. The $y$-update in \eqref{alg-1} then shows that \eqref{app-derivation-3.9} holds, which is just \eqref{app-derivation-2}.


\bibliographystyle{siamplain}
\bibliography{references_siam}
\end{document}